\documentclass[11pt]{amsart}
\usepackage{geometry}                
\usepackage{graphicx}
\usepackage{amsrefs}
\usepackage{amssymb}
\usepackage[matrix,arrow,ps]{xy}
\usepackage{hyperref}
\usepackage{epstopdf}
\usepackage{faktor}
\newcommand{\aut}{\mathfrak{hol}}

\newcommand{\C}{\mathbb{C}}
\newcommand{\CN}{\mathbb{C}^N}
\newcommand{\CNp}{\mathbb{C}^{N^\prime}}

\newcommand{\N}{\mathbb{N}}

\newcommand{\todo}[1]{}

\newlength{\extendaxesby}

\DeclareMathOperator{\id}{id}

 \newtheorem{thm}{Theorem}
\newtheorem{theorem}[thm]{Theorem}

\newtheorem{lemma}[thm]{Lemma}
\newtheorem{prop}[thm]{Proposition}

\newtheorem{cor}[thm]{Corollary}
\newtheorem{corollary}[thm]{Corollary}

\theoremstyle{definition}

\newtheorem{definition}[thm]{Definition}
\newtheorem{exa}{Example}
\newtheorem{example}[exa]{Example}

\newtheorem{remark}[thm]{Remark}

\begin{document}
\title{Local and infinitesimal rigidity \\of 
hypersurface embeddings}
\author{Giuseppe della Sala}
\address{Fakult\"at f\"ur Mathematik, Universit\"at Wien}
\email{giuseppe.dellasala@univie.ac.at}
\author{Bernhard Lamel}
\address{Fakult\"at f\"ur Mathematik, Universit\"at Wien}
\email{bernhard.lamel@univie.ac.at}
\author{Michael Reiter}
\address{Texas A\&M University at Qatar}
\email{reiter-michael@gmx.at}

 \thanks{The third author was supported by the FWF-Project I382 and QNRF-Project NPRP 7-511-1-098.}

\begin{abstract}
We study local rigidity properties of holomorphic
 embeddings of real hypersurfaces in $\C^2$ into
real hypersurfaces in $\C^3$ and show that infinitesimal conditions imply actual local 
rigidity in a number of (important) cases. We use this to show that generic embeddings
into a hyperquadric in $\C^3$ are locally rigid. 
\end{abstract}
\maketitle

\pagestyle{plain}

\section{Introduction}

Let $M\subset \mathbb C^N$, $M'\subset\mathbb C^{N'}$ be real hypersurfaces, $N\leq N'$. We say that $M$ admits a \emph{holomorphic embedding} into $M'$ if there exists a holomorphic map $H:\mathbb C^N\to \mathbb C^{N'}$, of maximal rank, such that $H(M)\subset M'$. 
The automorphism group $G={\rm Aut}(M)\times {\rm Aut}(M')$ acts on the set of 
embeddings of $M$ into $M'$ by the natural action $H\mapsto \sigma'\circ H \circ \sigma^{-1}$ for 
$(\sigma,\sigma')\in {\rm Aut}(M)\times {\rm Aut}(M')$. It may well happen that the orbit of this action on 
a given map $H$ does not give rise to all other possible embeddings of $M$ into $M'$. 
In analogy to 
the terminology of Euclidean geometry, we say  that an embedding of  $M$  into $M'$
 is \emph{rigid} (or simply that the embeddings of $M$ into $M'$ are 
\emph{rigid}) if there is only one orbit for the action of $G$.

The study of holomorphic embeddings of real submanifolds, especially with respect to rigidity, 
has been an active research field for several decades. In the case $N=N'$, the work of Chern and 
Moser \cite{CM} allowed in particular to investigate the automorphism group of Levi-nondegenerate 
hypersurfaces. When $N<N'$ the interplay between the automorphism group and the set of embeddings 
might be intricate, and it is of special interest to consider the cases where the automorphism group 
is large, i.e.\ the hyperquadric case. In \cite{We}, Webster showed that the immersions of a 
hypersurface germ $M\subset \mathbb C^N$ into the sphere $\mathbb S^{N+1}\subset \mathbb C^{N+1}$ are 
rigid if $N\geq 4$ (if $M$ is itself a sphere, this holds for $N\geq 3$). The rigidity property holds 
as well for maps $\mathbb S^N\to \mathbb S^{N'}$ with $N\geq 3$, $N'\leq 2N-2$ (see Faran \cite{Fa} 
and Huang \cite{Hu}), while it does not hold for $N'\geq 2N-1$: there are two inequivalent classes 
for $N'=2N-1$ (see \cite{HJ}) and infinitely many for $N'\geq 2N$ (see \cite{Da}). 

More recent work on rigidity of embeddings includes the so-called ``super-rigidity'' discovered
by Baouendi and Huang \cite{Baouendi:2005uq} and treated further in Baouendi, Huang, and 
Zaitsev \cite{Baouendi:2008cz}. Further more recent results in the ``low codimension'' case for  
 the sphere have been obtained by Ebenfelt, Huang, and Zaitsev 	\cite{Ebenfelt:2004um,Ebenfelt:2005wj}, and in the case of small signature difference, by Ebenfelt and Shroff \cite{Ebenfelt:2015wq}.

In low source dimension, where any codimension and any sign difference is large,  
the structure of the set of embeddings can be more complicated. For $N=2$, $N'=3$, it has been shown 
by Faran \cite{Fa2} that the immersions $\mathbb S^2\to\mathbb S^3$ are not rigid with respect to the 
group $G={\rm Aut}(\mathbb S^2)\times {\rm Aut}(\mathbb S^3)$, but they consist of only four 
different classes. The case of immersions of $\mathbb S^2$ in the hyperquadric of $\mathbb C^3$ with 
signature $(1,1)$ has been first treated in \cite{Le}, where it is shown that they consist of  exactly 
seven classes. 

In \cites{Re2}, the third author reproved the  results of Faran and Lebl  
by means of a parametrization method for nondegenerate mappings (introduced in \cite{La}), which 
 also allowed to study properties of the action of the 
\emph{isotropy group} $G_0$
(consisting of only those automorphisms which fix a given pair of points). Additionally in \cite{Re3}, the 
moduli space of 
the set
of embeddings $\mathcal{F}$ with respect to the  
action of $G$ (resp. $G_0$) has been studied from a topological point of view. As it turns out, the 
topology of $\mathcal F / G$ is discrete in the case of Faran, where
 the target is a sphere,  while it is 
not discrete in the case of Lebl, where the target
has signature $(1,1)$,  
 despite the fact that $\mathcal F / G$ is a finite set in both cases.

On the other hand, the quotient space $\mathcal F / G_0$ is Hausdorff in both cases, but it is no longer finite. 
 Thus, if we consider only the action of the isotropies, the family $\mathcal F$ is not rigid even from a \emph{local} point of view.

In the present paper, our goal is to study this notion of local rigidity (to be defined below) 
for immersions between germs of real-analytic hypersurfaces $M\subset\mathbb C^2$,
 $M'\subset \mathbb C^3$ passing through the origin. Before stating our results, we recall that
 the {\em automorphism group} ${\rm Aut}_0 (M)$ of a germ of a real-analytic hypersurface in $\mathbb C^N$ passing through $0$ is defined to be
 the space of all germs of biholomorphic maps $\sigma\colon (\CN, 0) \to (\CN,0)$ which satisfy $\sigma(M) \subset M$.
 We 
consider the action of the {\em isotropy group}
 $G_0={\rm Aut}_0(M)\times {\rm Aut}_0(M')$ on the space $\mathcal F_2$ of $2$-nondegenerate, transversal
  embeddings $H\colon (M,0) \to (M',0)$ and say that a map $H$ is {\em locally rigid} if it projects to an isolated  point in $\faktor{\mathcal F}{G_0}$ (we refer the reader to Section~2 for the definitions of nondegeneracy and 
  transversality we use).

We say that a holomorphic section $V$ of $T^{1,0}(\mathbb C^3)|_{H(\mathbb C^2)}$, 
vanishing at $0$, is an \emph{infinitesimal deformation} of $H$ 
if the real part of $V$ is tangent to $M'$ along $H(M)$ (for the formal definition, see Definition \ref{def:infdef}). We denote the real vector space of these vector fields by $\mathfrak {hol}_0 (H)$. We then 
have the following result which shows that if a map $H$ is  infinitesimally trivial with 
respect to isotropies, i.e. 
if $\dim \mathfrak {hol}_0 (H) = 0  $, we have in particular local rigidity:
\begin{theorem}\label{suffconintro}
Let $M$ be a germ of a strictly pseudoconvex real-analytic hypersurface through $0$ in $\C^2$, and $M'$ be 
 a germ of a real-analytic hypersurface in $\C^3$. Let $H\colon (\C^2,0) \to (\C^3,0)$ be a germ 
 of a $2$-nondegenerate, transversal embedding satisfying
  $\dim_{\mathbb R} \mathfrak {hol}_0 (H) = 0$. Then $H$ is an isolated point in $\mathcal F_2$, and 
  in particular, $H$ is locally rigid. 
 \end{theorem}

If $M'$ satisfies the stonger assumption of Levi-nondegeneracy, we 
can considerably strengthen this type of result by considering the notion of {\em infinitesimal rigidity}, meaning 
that all infinitesimal deformations $H$ come from infinitesimal automorphisms of $M'$. 
Recall that  $\mathfrak {hol}_0 (M')$ denotes the space of \emph{infinitesimal isotropies of $M'$}, 
consisting of holomorphic vector fields of $\mathbb C^3$, 
vanishing at $0$, whose real part is tangent to $M'$.

\begin{theorem}\label{suffcon2intro}
Let $M,M'$ be as in Theorem \ref{suffconintro}, and assume in addition that $M'$ is Levi-nondegenerate. Let $H\colon (\C^2,0) \to (\C^3,0)$ be a germ of a $2$-nondegenerate, transversal embedding satisfying  $\dim_{\mathbb R} \mathfrak {hol}_0 (H) = \dim_{\mathbb R}\mathfrak{hol}_0(M') $. Then $H$ is locally rigid.
\end{theorem}

We note that under reasonable assumptions on $M'$ (see \cite{JL}), which are
always satisfied in the cases we study,
 ${\rm Aut}_0(M')$ is a finite dimensional Lie group whose Lie algebra 
 is given by $\mathfrak {hol}_0 (M')$, which therefore is also finite dimensional. 
 The elements of $\mathfrak {hol}_0 (M')$ trivially restrict to elements of $\mathfrak {hol}_0 (H)$. With
 the assumptions of Theorem~\ref{suffcon2intro}, we have that this restriction map is injective, and hence the 
 inequality  $\dim \mathfrak {hol}_0 (H)\geq \dim\mathfrak {hol}_0 (M')$ always holds. 

Let us note that the concept of infinitesimal deformations has  already been considered by Cho and Han in \cite{CH}, where finite determination results are obtained by the method of complete systems when $M\subset \mathbb C^N$ and $M'\subset \mathbb C^{N'}$ are Levi-nondegenerate hypersurfaces and $H$ is a finitely nondegenerate embedding. We give a natural generalization of their result 
in Corollary~\autoref{cor:chohangeneral}.

We would like to point out to the reader 
 that the relationship between the space of infinitesimal deformations of $H$ and the 
 set of actual maps close to $H$ is not straightforward: among other things, it might 
 happen that the set of immersions is not a smooth manifold (see the structural results in \cite{Re3}). 
 The proof of Theorem \ref{suffcon2intro} is based on the analysis of the properties of the action of $G_0$ 
 on the space of jets of maps $M\to M'$ 
 (which is an analytic set by known parametrization results, see Theorem \ref{jetparam}), as well as a jet-parametrization result for the space $\mathfrak {hol}_0 (H)$. 
 We obtain the latter in Proposition \ref{linjetparam}: we remark that, for our purposes, 
 we need to take in account the dependence of the vector fields in
  $\mathfrak {hol}_0 (H)$ on their $4$-jet at $0$, as well as on the map $H$ and the source manifold $M$, thus we 
  need a result stronger than jet determination.

Due to the linearity of the sufficient conditions in Theorem \ref{suffconintro} and Theorem \ref{suffcon2intro}, 
their application to concrete situations is computationally easier than the study of the actual mapping problem. 
We use this to prove our  
Theorem \ref{genericity}, which  shows that a generic (in a sense specified in the statement of the result) embedding of a generic  hypersurface germ $M\subset \mathbb C^2$ into $\mathbb H^3_\pm$ is locally rigid. In order to prove this, we first compute the space of infinitesimal deformations of a special embedding $H_0$ of a certain hypersurface $M_0$ into $\mathbb H^3_\pm$ (cf. Example \ref{modelexample}), showing that its elements reduce to the restrictions of the elements of $\mathfrak {hol}_0 (\mathbb H^3_\pm)$ to $H_0(M_0)$. Then, we show that the method employed (following the same steps as in the proof of Proposition \ref{linjetparam}) allows to draw the same conclusion for a generic embedding, thus obtaining local rigidity. 

In Example \ref{itworks}, we apply the same method to an embedding of hypersurfaces which are both non-spherical (in such a situation, the computations involving the actual mapping equation seem to be even more complicated). On the other hand, the application of the same procedure in Example \ref{itdoesntwork} produces a space of infinitesimal deformations whose dimension is strictly larger than $\dim \mathfrak{hol}_0 (M')$. In general, 
 the dimension of $\mathfrak {hol}_0 (H)$ does not correspond to the actual dimension of the space of embeddings in a neighborhood of $H$: indeed,   an element $V\in \mathfrak {hol}_0 (H)$ which does not come from $\mathfrak{hol}_0 (M')$,  cannot be integrated.  In section \ref{spherecase}, we check that $\dim \mathfrak {hol}_0 (H)$ can be in fact much larger by computing the infinitesimal deformation space of an embedding of spheres which is known to be locally rigid by the work in \cite{Re2} and \cite{Re3}. In Example \ref{nonsphericalfamily} we construct a strictly pseudoconvex, non-spherical hypersurface which admits embeddings, which are not locally rigid.

The paper is organized as follows. In section \ref{sec:notation}, we introduce some notation and recall some well-known definitions. Furthermore, we give formal definition of local rigidity and summarize some known results related to the parametrization of nondegenerate embeddings (see Theorem \ref{jetparam}). In section \ref{sec:action}, we analyze the action of the isotropy group on the space of jets of (transversal, nondegenerate) maps $\mathbb C^2\to \mathbb C^3$, more specifically, with respect to properness and freeness. In section \ref{sec:linear} we define the notion of infinitesimal deformations $\mathfrak {hol}_0 (H)$ and prove Theorems \ref{suffcon1} and \ref{suffcon2},
which are 
reformulations of Theorems \ref{suffconintro} and \ref{suffcon2intro}, using a jet parametrization result for $\mathfrak {hol}_0 (H)$ obtained in section \ref{linearizedproblem}. In section \ref{sec:examples}, we apply the same methods to compute the space of infinitesimal deformations in several concrete examples of embeddings using Mathematica 9.0.1.0 \cite{wolfram}. In section \ref{sec:genericity}, we prove Theorem \ref{genericity} by a detailed examination of the parametrization procedure, using the computations performed in the model cases. In section \ref{spherecase}, we show that the sufficient condition provided by Theorem \ref{suffcon2} is not necessary, by computing the infinitesimal deformations of a certain locally rigid mapping between spheres.

\section{Notation, definitions}\label{sec:notation}

\subsection{Spaces of maps}

\

We will be interested in locally defined holomorphic maps from $\mathbb C^2$ to $\mathbb C^3$. From now on, we fix coordinates $Z = (z,w)$ for $\mathbb C^2$ and $Z'=(z_1', z_2', w')$ for $\mathbb C^3$.

We will identify the space of formal maps $H:\mathbb C^2 \to \mathbb C^3$ with $(\mathbb C[[z,w]])^3$, where $\mathbb C[[z,w]]$ is the ring of formal power series with complex coefficients in the indeterminates $(z,w)$. For any $k\in \mathbb N$ and $H\in (\mathbb C[[z,w]])^3$, we denote by $\varrho_k(H)$ the maximum of the moduli of the coefficients of the truncation of $H$ to the $k$-th order. Each $\varrho_k$ is a seminorm on $(\mathbb C[[z,w]])^3$, and it is easy to check that the collection of these seminorms induces on it a Frech\'et space topology.
 
 The tangent space to $(\mathbb C[[z,w]])^3$ at any of its points is of course again isomorphic to $(\mathbb C[[z,w]])^3$; nevertheless we will prefer to write an element $V$ of the tangent space as a  \lq\lq formal vector field of $\mathbb C^3$ defined along $\mathbb C^2$\rq\rq\ as follows:
\[V=\alpha(z,w)\frac{\partial}{\partial z_1'} + \beta(z,w)\frac{\partial}{\partial z_2'} + \gamma(z,w)\frac{\partial}{\partial w'}\]
where $\alpha,\beta,\gamma\in \mathbb C[[z,w]]$.

For any $R>0$, we denote by $B_R(0)$ the ball of $\mathbb C^2$ of radius $R$ centered at $0$. Furthermore we denote by ${\rm Hol}(\overline{B_R(0)}, \mathbb C^3)$ the space of holomorphic maps $B_R(0)\to \mathbb C^3$ which are continuous up to $\overline{B_R(0)}$. This is a Banach space with the sup norm on $\overline {B_R(0)}$.

We identify the space of germs at $0$ of holomorphic maps $\mathbb C^2\to \mathbb C^3$ with $(\mathbb C\{z,w\})^3$, where $\mathbb C\{z,w\}$ is the ring of convergent power series in the indeterminates $(z,w)$. We have that $(\mathbb C\{z,w\})^3=\cup_{R>0} {\rm Hol}(\overline{B_R(0)}, \mathbb C^3)$. Since for any $R'<R$ the restriction map ${\rm Hol}(\overline{B_R(0)}, \mathbb C^3)\to {\rm Hol}(\overline{B_{R'}(0)}, \mathbb C^3)$ is a compact operator, the direct limit topology gives to $(\mathbb C\{z,w\})^3$ the structure of a (DFS) space. In what follows we will refer to this topology for the germs of holomorphic maps $\mathbb C^2 \to \mathbb C^3$, as well as in general for all the rings of convergent power series. 

It is clear that the inclusions of ${\rm Hol}(\overline{B_R(0)}, \mathbb C^3)$ and $(\mathbb C\{z,w\})^3$ into $(\mathbb C[[z,w]])^3$ are continuous. We will identify the tangent spaces of ${\rm Hol}(\overline{B_R(0)}, \mathbb C^3)$ and $(\mathbb C\{z,w\})^3$ with holomorphic vector fields of $\mathbb C^3$ defined on a neighborhood of $0$ in $\mathbb C^2$, in a similar fashion as in the formal case.

\begin{remark}\label{implfunct}
With 
the direct limit topology on the rings of convergent power series, we have for instance that the maps $\mathfrak c_j:\mathbb C\{z,w\}\to \mathbb C\{z\}$ given by $\mathfrak c_j(f(z,w)=\sum_\ell f_\ell(z)w^\ell) = f_j(z)$ are continuous, uniformly for $j\in \mathbb N$. Moreover, denoting by $\Omega\subset(\mathbb C\{z,w\})^2$ the open subset given by germs with non-vanishing Jacobian at $0$, the inverse mapping theorem induces a continuous map $\Omega\to \Omega$. 
This implies that the series obtained by an application of the implicit function theorem also depend continuously on the initial data: we will use these observations in Proposition \ref{linjetparam}.
\end{remark}

\subsection{Embeddings into hypersurfaces}

We will fix a real-analytic hypersurface $M'=\{\rho=0\}\subset \mathbb C^3$, and denote by $\rho_{Z'}$ the complex gradient of $\rho$, i.e.\ 
\[\rho_{Z'} = \left (\frac{\partial \rho}{\partial z_1'}, \frac{\partial \rho}{\partial z_2'}, \frac{\partial \rho}{\partial w'}\right).\]
Up to an affine coordinate change, we can (and will) always assume that $0\in M'$ and $T_0(M')=\{{\rm Im}w' = 0\}$. We will be particularly interested in the spherical case, that is $M'=\mathbb H^3 = \{ \rho = {\rm Im} w' - |z_1'|^2 - |z_2'|^2 = 0\}\subset \mathbb C^3$. In this case we have
\[\rho_{Z'}= (-\overline z_1', - \overline z_2', - i/2).\]
We let $M\subset \mathbb C^2$ be a germ at $0$ of a strongly pseudoconvex hypersurface. We will be interested in studying the set of holomorphic embeddings of $M$ into $M'$ which satisfy certain generic conditions. The first one is the following:

\begin{definition}
We say that a map $H:\mathbb C^2 \to \mathbb C^3$ such that $H(0)=0$ is \emph{transversal} 
(to $ \{w = 0 \} $) at $0$ if $dH_0(\mathbb C^2) \not \subset \langle \frac{\partial}{\partial z_1'}, \frac{\partial}{\partial z_2'}\rangle$. If we write the components of $H$ as $(H^1,H^2,H^3)$, this is equivalent to $(\frac{\partial H^3}{\partial z}(0), \frac{\partial H^3}{\partial w}(0)) \neq (0,0)$. 
Its geometric meaning is that the differential of $H$ does not map 
the tangent space of $M$  into the complex tangent space of $M'$.
\end{definition}
\begin{remark}
If the map $H$ is an embedding of $M$ into $\mathbb H^3$ (or any strongly pseudoconvex hypersurface $M'$ of $\mathbb C^3$), the transversality condition is actually automatically satisfied. Indeed, let $N=H(M)\subset M'$; should $H$ not be transversal at $0$, we would have $\mathbb C T_0(N) = \langle \frac{\partial}{\partial z_1'}, \frac{\partial}{\partial z_2'}\rangle_{\mathbb C} = \mathbb C T_0^c(M')$. But then the commutator $[V,W]$ of any pair of vector fields $V\in T^{1,0}(N)$, $W\in T^{0,1}(N)$ would satisfy $[V,W](0)\in \mathbb C T_0(N) = \mathbb C T_0^c(M')$, contradicting the strong pseudoconvexity of $M'$.   
\end{remark}

\begin{lemma}\label{coord}
Let $M$ and $M'$ be as before and let $H:\mathbb C^2 \to \mathbb C^3$ be a local embedding of $M$ into $M'$ such that $H(0)=0$. Suppose that $H$ is transversal. Then, after a suitable holomorphic change of coordinates around $0$ in $\mathbb C^2$ (and, possibly, after composing with the map $\mathbb C^3\ni (z_1',z_2',w')\to (z_2',z_1',w')\in \mathbb C^3$), $H$ can be locally expressed  as $H(z,w)=(z, F(z,w), w)$ for a certain germ of holomorphic function $F: \mathbb C^2\to \mathbb C$ such that $F(0) = 0$.
\end{lemma}
\begin{proof}
The image $\mathcal H = H(\mathbb C^2)$ of $\mathbb C^2$ is a complex hypersurface of $\mathbb C^3$ in a neighborhood of $0$, whose tangent space at $0$ is not contained in $\langle \frac{\partial}{\partial z_1'}, \frac{\partial}{\partial z_2'}\rangle$ by the transversality assumption. Writing $H=(H_1,H_2,H_3)$, the previous remark and the fact that $H$ is of rank $2$ imply (up to exchanging $z_1'$ and $z_2'$) that
\[ {\rm det}\left( \begin{array}{cc}
\frac{\partial H_1}{\partial z}(0) & \frac{\partial H_3}{\partial z}(0) \\
\frac{\partial H_1}{\partial w}(0) & \frac{\partial H_3}{\partial w}(0)
\end{array} \right) \neq 0.
\]

It follows that the map $(z,w)\to (H_1(z,w), H_3(z,w))$ admits a local inverse $\Phi$ such that $\Phi(0)=0$. The conclusion follows by composing $H$ with $\Phi$.
\end{proof}
\begin{remark}\label{compose}
In the case $M'=\mathbb H^3_{\varepsilon}$, performing the change of coordinates of the lemma above implies that $M$ is expressed as
\begin{equation} \label{oftheform}
M = \{ {\rm Im}w - |z|^2 - \varepsilon |F(z,w)|^2 = 0\}
\end{equation}
for a certain germ of holomorphic function $F: \mathbb C^2\to \mathbb C$. 

\end{remark}

 \subsection{Finitely nondegenerate maps and their jet parametrization}
 
 Denote by $L$ a CR vector field tangent to $M$ around $0$. For instance, if $M$ is written as in (\ref{oftheform}) and $F(z,w)=F(z)$ does not depend on $w$, we can choose
\[L = \frac{\partial}{\partial \overline z} - 2i(z + \varepsilon F(z)\overline{F_z(z)})\frac{\partial}{\partial \overline w}.\]
 
Let $H:\mathbb C^2\to \mathbb C^3$ be a map. 
We recall a definition first introduced in \cite{La}:

\begin{definition}\label{def:nondeg}
Given $k_0\in \mathbb N$, the map $H$ is said to be \emph{$k_0$-nondegenerate} at $0$ if, defining
\[E_k(p) = {\rm span} \left\{L^{j}\rho_{Z'}(H(Z), \overline{H(Z)})\biggr|_{Z=p} : 0\leq j \leq k \right\}, \]
we have $E_{k_0}(0) = \mathbb C^3$ and 
$E_{k_0-1}(0) \neq \mathbb C^3$. We remark that, when $M$ and $M'$ are given, the notion of $k_0$-nondegeneracy at $0$ only depends on the $k_0$-jet of $H$ at $0$. Let us remark 
that even though the definition of a $k_0$-nondegenerate map is given for an 
arbitrary map $H$, it is only invariant under biholomorphic changes of coordinates if 
$H(M)\subset M'$.
\end{definition}

\begin{remark}
It is worth remarking that if we are dealing with nondegenerate {\em embeddings} $H$,
we can also take the view of $H$ giving rise to a submanifold $H(M)\subset M'$, 
which furthermore, if $H$ is transversal, is of the form $H(M) = M' \cap V $
for some complex submanifold $V\subset \CNp$. The above definition then gives 
rise to a notion of {\em nondegenerate CR submanifold} of $M'$. We will, however,
mostly prefer to start with an embedding (because some of our results will
not be restricted to transversal embeddings).
\end{remark}


We define  
\[J_0^k = \faktor{\mathfrak{m} \mathbb C \{z,w\}^3 }{\mathfrak{m}^{k+1}}, \]
where $\mathfrak{m}=(z,w)$ is the maximal ideal, the space of the $k$-jets $\Lambda$ 
of holomorphic maps $\mathbb C^2\to \mathbb C^3$ such that $\Lambda(0)=0$, with the 
natural projection $j_0^k$.
 For a given $k$, $\Lambda$ both denotes the coordinates in $J_0^k$ and the associated polynomial map. More precisely we will write
$\Lambda = (\Lambda_1,\Lambda_2,\Lambda_3)$ with $\Lambda_j=(\Lambda_j^{m,\ell})$, $0\leq m+\ell \leq k$, where
\[ \Lambda = j_0^k H \text{ if  and only if } \Lambda_j^{m,\ell} = \frac{1}{m!\ell!} 
\frac{\partial^{m+\ell} H_j}{\partial z^m \partial w^\ell} (0),  \]
and 
\[ \Lambda (z,w) =(\Lambda_1 (z,w),\Lambda_2 (z,w),\Lambda_3 (z,w))  \]
with
\[ \Lambda_j (z,w) = \sum_{m,\ell} \Lambda_j^{m,\ell} z^m w^\ell.\]
We also define the open subset 
$\widetilde J_0^k \subset J_0^k$ 
by 
\[\widetilde J_0^k = \left \{ \Lambda\in J_0^k \colon \begin{vmatrix}
	\Lambda_1^{1,0} & \Lambda_2^{1,0} \\ \Lambda_1^{2,0} & \Lambda_2^{2,0}
\end{vmatrix}  \neq 0 ,\, \Lambda_3^{0,1} \neq 0 \right\} \]
and think of it as the  
$k$-jets of transversal, $2$-nondegenerate maps (whose jets are actually all contained 
in $\widetilde J_0^k$ provided that $M'$ is Levi-nondegenerate). 

We will denote by $\mathcal F$ the space of transversal germs of holomorphic maps $H:\mathbb C^2\to \mathbb C^3$ such that $H(0) = 0$ and $H(M)\subset M'$ and its open subset of $2$-nondegenerate 
maps by $\mathcal F_2$. 
We equip $\mathcal F$ and $\mathcal F_2$ with their natural 
 topologies as  subsets of $(\mathbb C\{z,w\})^3$.

The following result is a consequence of \cite[Proposition 25, Corollaries 26, 27]{La} combined with 
 \cite[Theorem 5]{JL}; see also \cite[Lemmas 5.7, 5.8]{Re}.

\begin{theorem}
\label{jetparam}  Let $M\subset \mathbb C^2$ be the germ of a strongly pseudoconvex, real-analytic hypersurface, $0\in M$, and let $M'\subset \mathbb C^3$ be a real-analytic hypersurface germ. 
There exists a polynomial
$q(\Lambda, \bar \Lambda)$ on $J_0^4$,
an open set $\mathcal U \subset \mathbb C^2 \times \{q\neq 0\}$ containing $(0,j_0^4 H)$
for every $H\in\mathcal F_2$, and a holomorphic map
$\Phi \colon \mathcal U \to \C^3 $ satisfying $\Phi(0,\Lambda) = 0$, which 
can be written as 
\[ \Phi (Z,\Lambda) = \sum_{\alpha\in \N^2} \frac{p_\alpha (\Lambda, \bar \Lambda)}{q(\Lambda, \bar \Lambda)^{d_\alpha}} Z^\alpha, \quad  p_\alpha, q \in \C[\Lambda, \bar \Lambda], \quad d_\alpha\in\N \] such that
\[ H(Z) = \Phi (Z,j_0^4 H), \quad H\in \mathcal F_2.\]

In particular, there exist (real) polynomials $c_j$, $j\in\N$ on $J_0^4$ such that
\[ j_0^4 \mathcal F_2 = \{ \Lambda\in J_0^4 \colon q(\Lambda, \bar \Lambda)\neq 0, \, c_j (\Lambda, \bar \Lambda) =0 \}.\]
\end{theorem}

\subsection{Isotropy group} \label{isogro}

We will denote by ${\rm Aut}_0 (M)$ the group of the germs $\sigma$ of biholomorphic maps $\mathbb C^2\to \mathbb C^2$, defined around $0$, such that $\sigma(0) = 0$ and $\sigma(M)\subset M$. We define the group ${\rm Aut}_0(M')$ in an analogous way. Set $G=  {\rm Aut}_0 (M) \times {\rm Aut}_0 (M') $; we will refer to $G$ as the \emph{isotropy group}. The group $G$ acts on $\mathcal F$ in the following way: given $(\sigma, \sigma')\in G$ we define a map $\mathcal F \to \mathcal F$ by
\[\mathcal F \ni H \to \sigma' \circ H \circ \sigma^{-1} \in \mathcal F.\]
Note that the notion of $k$-nondegeneracy is invariant under holomorphic changes of coordinates (see 
\cite[Lemma 14]{La}), hence for any $H\in \mathcal F_2$ we have that $\sigma' \circ H \circ \sigma^{-1}$ is  an element of
 $\mathcal F_2$, and the action of $G$ therefore restricts onto $\mathcal F_2$.

 Our general aim is to study the structure of the quotient $\mathcal N = \faktor{\mathcal F_2}{G}$ 
under this group action, which we endow with its natural topology as a quotient space. 
More precisely, we want to study the following property:
\begin{definition}\label{locrig}
Let $M$ and $M'$ be germs of hypersurfaces in $\mathbb C^2$ (resp. $\mathbb C^3$) around $0$, 
and let $H\in \mathcal F$ be a transversal embedding of $M$ into $M'$. 
We say that $H$ is \emph{locally rigid} if $H$ projects to an isolated point in the 
quotient $\faktor{\mathcal F}{G}$.
\end{definition}

\begin{remark}\label{rem:equcon}
 $H:M\to M'$ is locally rigid according to the definition above if and only if  there exists a neighborhood $U$ of $H$ in $(\mathbb C\{z,w\})^3$ such that for every $\hat H\in \mathcal F\cap U$ there is $g\in G$ such that $\hat H=g H$. In other words, $H$ is locally rigid if and 
 only if all the maps in $\mathcal F$ which are close enough to $H$ are equivalent to $H$. 

Indeed, by the definition of the quotient topology, $H$ projecting to an isolated point in the quotient amounts to the existence of a $G$-invariant neighborhood $\mathcal U$ of $H$ in $(\mathbb C\{z,w\})^3$ such that $\mathcal F\cap \mathcal U = G\cdot H$ (where we denote by $G\cdot H$ the orbit of $G$ through $H$), and thus local rigidity is in principle a stronger property. On the other hand, let $U$ be any neighborhood of $H$ such that $\mathcal F \cap U = (G\cdot H)\cap U$: since $G$ acts on $(\mathbb C\{z,w\})^3$ by homeomorphisms, the saturation $\mathcal U=\bigcup_{g\in G} gU$ of $U$ is a ($G$-invariant) open set. Given any $\hat H\in \mathcal F\cap \mathcal U$, we have $\hat H\in \hat g U$ for some $\hat g\in G$, which implies that $\hat g^{-1} \hat H\in \mathcal F\cap U$, hence by assumption $\hat g^{-1}\hat H\in G\cdot H$. It follows that in turn $\hat H\in G\cdot H$, which proves the equivalence of the two definitions.

\end{remark}

Because of the aforementioned parametrization results, we will also be interested in the action of $G$ on the jet-space $J_0^4$. 
 Recall that we are
treating an element $\Lambda$ of $J_0^4$ as a polynomial map $\mathbb C^2\to \mathbb C^3$; we  define the action of $(\sigma,\sigma')\in G$ on  $J_0^4$ by
\[J_0^4 \ni \Lambda \to j_0^4 (\sigma' \circ \Lambda \circ \sigma^{-1}) \in J_0^4. \]
If $M$ is strictly pseudoconvex and  $M'$ is Levi-nondegenerate, one can check that this action
 preserves the space $\widetilde J_0^4$.

It is a classical fact (see \cite{CM}) that ${\rm Aut_0}(M)$ and ${\rm Aut_0}(M')$ are finite dimensional Lie groups as long as $M,M'$ are Levi-nondegenerate. Furthermore, if $M,M'$ are strongly pseudoconvex then ${\rm Aut_0(M)}$ (resp. ${\rm Aut_0}(M')$) is always compact unless $M \cong \mathbb H^2 = \{{\rm Im} w - |z|^2 = 0\}$ (resp. $M'\cong \mathbb H^3_+ = \{{\rm Im} w' - |z_1'|^2 - |z_2'|^2 = 0\}$), see \cite{BV}. If $M'$ is Levi-nondegenerate but the eigenvalues of its Levi form are of opposite signs, ${\rm Aut_0}(M')$ is isomorphic to a subgroup of $U(1,1)$ unless $M'\cong \mathbb H^3_- = \{{\rm Im} w' - |z_1'|^2 + |z_2'|^2 = 0\}$, see \cite{Ez}.
Since we are especially interested in the spherical case, we take a closer look at the groups ${\rm Aut_0}(\mathbb H^2),{\rm Aut_0}(\mathbb H^3_\pm)$, which are well-known: the following explicit description is borrowed from \cite{Re}.

Consider $\Gamma = \mathbb R^+ \times \mathbb R \times \mathbb S^1 \times \mathbb C$ as a parameter space. The map
\begin{equation}\label{paramh2}
\Gamma \ni \gamma =  (\lambda, r, u, c) \to \sigma_{\gamma}(z,w) = \frac{(\lambda u (z + c w), \lambda^2 w) }{1-2i \overline c z+ (r - i |c|^2 )w} \in {\rm Aut_0}(\mathbb H^2)
\end{equation}
is a diffeomorphism between $\Gamma$ and ${\rm Aut_0}(\mathbb H^2)$, where $u\in \mathbb S^1=\{e^{it}:t\in \mathbb R\}\subset\mathbb C$.

Similarly, consider $\Gamma'_\varepsilon = \mathbb R^+ \times \mathbb R  \times U_\varepsilon \times \mathbb C^2$ as a parameter space, where $U_\varepsilon = U(2)$ if $\varepsilon=+1$ and $U_\varepsilon=U(1,1)$ if $\varepsilon=-1$. Then the map
\begin{equation}\label{param}
\Gamma'_\varepsilon \ni \gamma' =  (\lambda', r', U', c') \to \sigma'_{\gamma'}(z',w') = \frac{(\lambda' U' \ {}^t(z' + c' w'), \lambda'^2 w') }{1-2i \langle \overline c', z' \rangle_\varepsilon + (r' - i \|c'\|^2_\varepsilon )w'} \in {\rm Aut_0}(\mathbb H^3_\varepsilon)
\end{equation}
is a diffeomorphism between $\Gamma'_\varepsilon$ and ${\rm Aut_0}(\mathbb H^3_\varepsilon)$: here we denote by $\langle \cdot, \cdot \rangle_\varepsilon$ the product on $\mathbb C^2$ given by $\langle z', \widetilde z' \rangle_\varepsilon= z_1'\widetilde z_1' + \varepsilon z_2'\widetilde z_2'$ and we write $\| z' \|_\varepsilon^2 = \langle \overline z', z' \rangle_\varepsilon$.

The Lie algebra $\mathfrak{hol}_0(\mathbb H^3_\pm)$ of ${\rm Aut}_0(\mathbb H^3_\pm)$ is given by the infinitesimal automorphisms of $\mathbb H^3_\pm$ which vanish at $0$, i.e.\ by the holomorphic vector fields $Z$, defined in a neighborhood of $0$ in $\mathbb C^3$, such that $Z(0)=0$ and ${\rm Re}Z$ is tangent to $\mathbb H^3_\pm$. For instance, a parametrization of $\mathfrak{hol}_0(\mathbb H^3)$ is given by
\[ \left ( (t+ih_{11})  z_1' + h_{12}z_2'  + i\frac{\overline b_1}{2}w' +  b_1 {z_1'}^2 + b_2 z_1' z_2'  + s z_1' w'\right )\frac{\partial}{\partial z_1'} + \]
\[ + \left( (t+ih_{22}) z_2' - \overline h_{12}z_1' +  i\frac{\overline b_2}{2}w' + b_1 z_1' z_2' + b_2 {z_2'}^2 + s z_2' w'\right ) \frac{\partial}{\partial z_2'} + \]
\[ +  (2tw' + b_1 z_1' w' + b_2 z_2' w'  + s {w'}^2)\frac{\partial}{\partial w'},\]
where $(t,h_{11},h_{22},s,h_{12},b_1,b_2)$ belong to the parameter space $\mathbb R^4\times \mathbb C^3$.

\section{The action of the isotropy group on the jet space}\label{sec:action}

In this section we study the properties (properness and freeness) of the action of the isotropy group $G = {\rm Aut}_0(M) \times {\rm Aut}_0(M')$ on the jet space $\widetilde J^4_0$, where we assume both $M$ and $M'$ to be Levi-nondegenerate. This allows to recover the corresponding properties for the action of the isotropy group on the space of maps $\mathcal F_2$, since the two actions are conjugated (cf. Lemma \ref{Ginv}). This study has been carried out in \cite{Re} for the case $M = \mathbb H^2$, $M' = \mathbb H_\pm^3$ (which is the most difficult one), hence we will suppose that not both $M$ and $M'$ are biholomorphic to a hyperquadric. 

For technical reasons we need to restrict the action of $G$ to a particular subset of $\widetilde J_0^4$. For $\varepsilon\in \{-1,1\}$, let $E_\varepsilon$ be the subset of $\widetilde J_0^4$ defined by  
\[E_\varepsilon=\{\Lambda_3^{1,0}=\Lambda_3^{2,0} = 0,\  \Lambda_3^{0,1} = |\Lambda_1^{1,0}|^2 + \varepsilon |\Lambda_2^{1,0}|^2  \}.\]
One can easily see that $E_\varepsilon$ is a (real algebraic) submanifold of $\widetilde J_0^4$.

\begin{remark}\label{furthermore}
If coordinates are chosen such that $M=\{{\rm Im}\, w = |z|^2 +O(2)\}$, $M'=\{{\rm Im}\, w' = |z_1'|^2 + \varepsilon |z_2'|^2 + O(2)\}$, a straightforward computation shows that $E_\varepsilon$ contains the $4$-th jet of any map of $\mathcal F_2$.
\end{remark}

\begin{lemma}
The submanifold $E_\varepsilon$ is invariant under the action of $G_\varepsilon= {\rm Aut}_0(\mathbb H^2) \times {\rm Aut}_0(\mathbb H^3_\varepsilon)$.
\end{lemma}
\begin{proof}
Indeed, a computation shows that for any $g=(\sigma_{(\lambda, r, u, c)}, \sigma'_{(\lambda', r', U', c')})\in G_\varepsilon$, and any $\Lambda\in E_{\varepsilon}$ we have, putting $\widetilde \Lambda= g \Lambda$,
\begin{align*}
\widetilde \Lambda_3^{1,0} & = u\lambda{\lambda'}^2 \Lambda_3^{1,0}=0,\\
\widetilde \Lambda_3^{2,0} & = {\lambda'}^2(\lambda^2 u^2\Lambda_3^{2,0}-4i\lambda u \overline c \Lambda_3^{1,0})=0,\\
\widetilde \Lambda_3^{0,1} & = {\lambda'}^2\lambda^2 \Lambda_3^{0,1}, \ \ {}^t(\widetilde \Lambda_1^{1,0}, \widetilde \Lambda_2^{1,0}) = \lambda\lambda' u U' \ {}^t(\Lambda_1^{1,0}, \Lambda_2^{1,0}) \Rightarrow \\
\Rightarrow \widetilde \Lambda_3^{0,1} - \| {}^t(\widetilde \Lambda_1^{1,0}, \widetilde \Lambda_2^{1,0})\|_\varepsilon^2 & = {\lambda'}^2\lambda^2 (\Lambda_3^{0,1} - \|u U' \ {}^t(\Lambda_1^{1,0}, \Lambda_2^{1,0})\|_\varepsilon^2) \\
& = {\lambda'}^2\lambda^2 (\Lambda_3^{0,1} - \|\ {}^t(\Lambda_1^{1,0}, \Lambda_2^{1,0})\|_\varepsilon^2)=0,
\end{align*}
hence $\widetilde \Lambda\in E_\varepsilon$.
\end{proof}

We recall that the action of a topological group $\mathcal G$ on a space $X$ is called \emph{proper} if the map $\mathcal G\times X \ni (g,x) \to (x,gx)\in X\times X$ is proper.

\begin{lemma}\label{proact}
Suppose that $M\not \cong \mathbb H^2$ is strongly pseudoconvex and $M'=\mathbb H^3_\varepsilon$. Then the action of $G =  {\rm Aut}_0 (M) \times{\rm Aut}_0 (\mathbb H^3_\varepsilon)$ on $E_\varepsilon$ is proper.
\end{lemma}
\begin{proof}
It is sufficient to show the following:  let $N>1$, and let $\{(\widehat\Lambda_n,\widetilde\Lambda_n)\}_{n\in\mathbb N} \subset E_\varepsilon \times E_\varepsilon$, $\{g_n = (\sigma_n,\sigma'_n)\}_{n\in \mathbb N} \subset G$ be sequences such that $|\widehat\Lambda_n|, |\widetilde \Lambda_n|\leq N$, $|(\widehat\Lambda_n)_3^{0,1}|, |(\widetilde \Lambda_n)_3^{0,1}|\geq 1/N$  and $\widetilde\Lambda_n = g_n \widehat\Lambda_n$ for all $n\in \mathbb N$. Then $\{g_n\}$ admits a convergent subsequence. 

As already noted, the assumption that $M\not\cong \mathbb H^2$ implies that ${\rm Aut}_0(M)$ is a compact Lie group. It follows that, defining $\Lambda_n = \widehat\Lambda_n \circ \sigma_n$, we still have $|\Lambda_n|\leq N$ and $|(\Lambda_n)_3^{0,1}|\geq 1/N$ for all $n\in \mathbb N$ (where we choose a possibly larger $N>0$) and $\widetilde \Lambda_n = \sigma'_n \circ \Lambda_n$. Thus, what we need to prove is that the sequence $\{\sigma'_n\}$ is relatively compact in ${\rm Aut}_0(\mathbb H^3_\varepsilon)$. Using the parametrization (\ref{param}), it is enough to show that the preimage $\{\gamma'_n=(\lambda'_n, r'_n, U'_n, c'_n)\}$ of $\{\sigma'_n\}$ in the parameter space $\Gamma'_\varepsilon$ is relatively compact.

We first look at the third component of the first jet, obtaining that
\[{\lambda'_n}^2(\Lambda_n)_3^{0,1} = (\widetilde \Lambda_n)_3^{0,1},\]
and thus $\lambda'_n$ is bounded.
Looking at the first two components of the first jet we get the following equation
\begin{equation}\label{firsttwofirst2}
\lambda_n' U_n' \ {}^t((\Lambda_n)_1^{1,0}, (\Lambda_n)_2^{1,0}) = {}^t((\widetilde \Lambda_n)_1^{1,0}, (\widetilde \Lambda_n)_2^{1,0}).
\end{equation}
If we write $U'_n\in U_\varepsilon$ as 
\[U'_n = \left( \begin{array}{cc} u'_n a_{1,n} &  - \varepsilon u'_n a_{2,n} \\  \overline a_{2,n} & \overline a_{1,n} \end{array} \right) \ {\rm with}\  |u'_n|=1, \ |a_{1,n}|^2\ + \varepsilon |a_{2,n}|^2=1,\]
conjugating the second component of \eqref{firsttwofirst2} we can rewrite it as 
\[\lambda_n'  \left( \begin{array}{cc} u'_n  (\Lambda_n)_1^{1,0} & - \varepsilon   u'_n  (\Lambda_n)_2^{1,0} \\ ( \overline \Lambda_n)_2^{1,0} & (\overline \Lambda_n)_1^{1,0} \end{array} \right) \left(\begin{array}{c} a_{1,n}  \\  a_{2,n} \end{array} \right) =   \left(\begin{array}{c}  (\widetilde \Lambda_n)_1^{1,0} \\  (\overline{\widetilde \Lambda}_n)_2^{1,0}\end{array} \right).\]
Since $\Lambda_n \in E_\varepsilon$ the determinant of the matrix on the left-hand side is equal to
\[u_n' (|(\Lambda_n)_1^{1,0}|^2 +\varepsilon |(\Lambda_n)_2^{1,0}|^2) = u_n' (\Lambda_n)_3^{0,1},\]
and since $|(\Lambda_n)_3^{0,1}|\geq1/N$ we conclude that the sequence $(a_{1,n},a_{2,n})$ is bounded in $\mathbb C^2$.

Using the information gained up to now, we see from the first two components of the first jet
\[\lambda_n' U_n' \ {}^t((\Lambda_n)_1^{0,1}, (\Lambda_n)_2^{0,1}) +(\Lambda_n)_3^{0,1}\ {}^tc_n' = {}^t((\widetilde \Lambda_n)_1^{0,1}, (\widetilde \Lambda_n)_2^{0,1})\]
that $c_n'$ is bounded in $\mathbb C^2$.  Finally, we consider the third component of the second jet, which gives us the equation
\[(\widetilde \Lambda_n)_3^{0,2} = -{\lambda'}_n^4 (\Lambda_n)_3^{0,1} r'_n + R_n,\]
where $R_n$ is a polynomial expression in the second jet of $\Lambda_n$, in $\lambda'_n$ and in the coefficients of $c'_n$ and $U'_n$ (but which does not depend on $r'_n$). This shows that the sequence $r'_n$ is bounded in $\mathbb R$, and concludes the proof.
\end{proof}

Next, we consider the case $M=\mathbb H^2$. Here the proof can be carried out along broadly  analogous lines, but the computations do not actually reduce to the ones performed in the previous lemma.

\begin{lemma}\label{proact2}
The action of ${\rm Aut}_0 (\mathbb H^2) \times U_\varepsilon$ on $E_\varepsilon$ is proper.
\end{lemma}
\begin{proof} 
As in the previous lemma we show the following: let $N>1$, and let $\{(\Lambda_n,\widetilde\Lambda_n)\}_{n\in\mathbb N} \subset E_\varepsilon \times E_\varepsilon$, $\{g_n = (\sigma_n,\sigma'_n)\}_{n\in \mathbb N} \subset G$ be sequences such that $|\Lambda_n|, |\widetilde \Lambda_n|\leq N$, $|(\Lambda_n)_3^{0,1}|, |(\widetilde \Lambda_n)_3^{0,1}|\geq 1/N$
 and $\widetilde\Lambda_n = g_n \Lambda_n$ for all $n\in \mathbb N$. Then $\{g_n\}$ admits a convergent subsequence. 

We first look at the third component of the first jet, obtaining that
\[\lambda_n^2(\Lambda_n)_3^{0,1} = (\widetilde \Lambda_n)_3^{0,1},\]
which implies that the sequence $\lambda_n$ is bounded above and below.
Looking at the first two components of the first jet we get the following equation
\begin{equation}\label{firsttwofirst}
\lambda_n u_n U_n' \ {}^t((\Lambda_n)_1^{1,0}, (\Lambda_n)_2^{1,0}) = {}^t((\widetilde \Lambda_n)_1^{1,0}, (\widetilde \Lambda_n)_2^{1,0}).
\end{equation}
If we write $U'_n\in U_\varepsilon$ as 
\[U'_n = \left( \begin{array}{cc} u'_n a_{1,n} & -\varepsilon u'_n a_{2,n} \\ \overline a_{2,n} & \overline a_{1,n} \end{array} \right) \ {\rm with}\  |u'_n|=1, \ |a_{1,n}|^2\ + \varepsilon|a_{2,n}|^2=1,\]
conjugating the second component of \eqref{firsttwofirst} we can rewrite it as 
\[\lambda_n  \left( \begin{array}{cc} u'_n  (\Lambda_n)_1^{1,0} & -\varepsilon u'_n  (\Lambda_n)_2^{1,0} \\ ( \overline \Lambda_n)_1^{1,0} &  (\overline \Lambda_n)_2^{1,0} \end{array} \right) \left(\begin{array}{c} a_{1,n}  \\  a_{2,n} \end{array} \right) =   \left(\begin{array}{c} \overline u_n (\widetilde \Lambda_n)_1^{1,0} \\  u_n (\overline{\widetilde \Lambda}_n)_2^{1,0}\end{array} \right).\]
Since $\Lambda_n \in E_\varepsilon$ the determinant of the matrix on the left-hand side is equal to
\[u_n' (|(\Lambda_n)_1^{1,0}|^2 +\varepsilon |(\Lambda_n)_2^{1,0}|^2) = u_n' (\Lambda_n)_3^{0,1},\]
and since $|(\Lambda_n)_3^{0,1}|\geq1/N$ we conclude that the sequence $(a_{1,n},a_{2,n})$ is bounded in $\mathbb C^2$ (of course the previous computation is superfluous if $\varepsilon=+1$).

The remaining equations coming from the first two components of the first jet can be written as follows
\[\lambda_n u_n U_n'\left (\lambda_n\left(  \begin{array}{c} (\Lambda_n)_1^{0,1} \\   (\Lambda_n)_2^{0,1} \end{array} \right) +  u_n c_n\left(\begin{array}{c}  (\Lambda_n)_1^{1,0}\\ (\Lambda_n)_2^{1,0} \end{array}\right) \right)= \left(\begin{array}{c}  (\widetilde \Lambda_n)_1^{0,1} \\   (\widetilde \Lambda_n)_2^{0,1}  \end{array} \right), \]
therefore
\[ \frac{1}{\lambda_n  u_n} {U_n'}^{-1}\left(  \begin{array}{c} (\widetilde\Lambda_n)_1^{0,1} \\   (\widetilde\Lambda_n)_2^{0,1} \end{array} \right) - \lambda_n \left(\begin{array}{c} (\Lambda_n)_1^{0,1}\\ (\Lambda_n)_2^{0,1} \end{array}\right)=  u_n c_n \left(\begin{array}{c}  (\Lambda_n)_1^{1,0} \\   (\Lambda_n)_2^{1,0}  \end{array} \right) \]
which implies
\[ \frac{|c_n| }{N}\leq |c_n|  \left \| \left(\begin{array}{c}  (\Lambda_n)_1^{1,0} \\   (\Lambda_n)_2^{1,0}  \end{array} \right) \right \|_{\varepsilon} \leq \lambda_n \left \|\left(\begin{array}{c}  (\Lambda_n)_1^{0,1}\\  (\Lambda_n)_2^{0,1} \end{array}\right) \right \|_{\varepsilon} +\frac{1}{\lambda_n } \left \|\left( \begin{array}{c} (\widetilde\Lambda_n)_1^{0,1} \\   (\widetilde\Lambda_n)_2^{0,1} \end{array} \right) \right \|_{\varepsilon}\]
so that $c_n$ is bounded in $\mathbb C$. Finally, we consider the third component of the second jet, which gives us the equation
\[(\widetilde \Lambda_n)_3^{0,2} = -\lambda_n^2 (\Lambda_n)_3^{0,1} r_n + R_n,\]
where $R_n$ is a polynomial expression in the second jet of $\Lambda_n$ and in $\lambda_n$, $c_n$, $u_n$ (which does not depend on $r_n$). This shows that the sequence $r_n$ is bounded in $\mathbb R$, and concludes the proof.
\end{proof}

Let now $G'$ be the subgroup of $G$ consisting of elements which only act on the target space, i.e.\ $G' = \{ \id \} \times {\rm Aut}_0(M')$. Clearly, the action of $G$ restricts to an action of $G'$ on $\widetilde J_0^4$, which is still proper. The action of $G'$, however, is in addition free. In order to verify this (local) statement, it is convenient to perform a suitable change of coordinates: 

\begin{lemma} \label{coord2}
Let $M, M'$ and $H$ be as before, with $M'$ Levi-nondegenerate. Then there are local  changes of coordinates $\Phi,\Phi'$ of $\mathbb C^2$ and $\mathbb C^3$ such that 
\begin{itemize}
\item$\Phi'\circ H\circ \Phi(z,w)=(z,F(z,w),w)$ for a certain germ of holomorphic function $F: \mathbb C^2\to \mathbb C$ such that $F(0) = 0$;
\item the automorphism group of $M'$ at $0$ is a subgroup of ${\rm Aut}_0 (\mathbb H^3)$ (if $M'$ is strongly pseudoconvex) or ${\rm Aut}_0 (\mathbb H^3_-)$ (if the signature of $M'$ is $(1,1)$). 
\end{itemize}
\end{lemma}
\begin{proof}
Suppose that $M'$ is strongly pseudoconvex: by  \cite{KL}, with a suitable change of coordinates $\Phi'$ in $\mathbb C^3$ we can assume that ${\rm Aut}_0(M')$ is a subgroup of ${\rm Aut}_0(\mathbb H^3)$ (actually, a compact linear subgroup if $M'\neq \mathbb H^3$). If the signature of $M'$ is $(1,1)$, the analogous conclusion follows by invoking \cite{Ez}. We can then change coordinates in $\mathbb C^2$ via $\Phi$ as given by Lemma \ref{coord}  to achieve the first point: clearly the automorphism group of $M'$ is not affected by this.
\end{proof}

\begin{lemma}\label{free}
Let $\Lambda\in E_\varepsilon$ be the $4$-jet of a map of the form $(z,w)\to(z,F(z,w),w)$. Then the stabilizer of $\Lambda$ under the action of $G'_\varepsilon = \{\id\} \times {\rm Aut}_0(\mathbb H^3_\varepsilon)$ is trivial.
\end{lemma}
\begin{proof}
By assumption, we can write the projection of $\Lambda$ to its $2$-jet as
\[j_0^2(\Lambda):(z,w) \to (z, \Lambda_2^{1,0}z + \Lambda_2^{0,1}w + \Lambda_2^{2,0}z^2 + \Lambda_2^{1,1}zw + \Lambda_2^{0,2}w^2,w).\]
Since $\Lambda\in \widetilde J_0^4$ it follows that $\Lambda_2^{2,0}\neq 0$: indeed, we must have
\[0\neq \Lambda_1^{1,0}\Lambda_2^{2,0}- \Lambda_2^{1,0}\Lambda_1^{2,0} = 1\cdot \Lambda_2^{2,0} - \Lambda_2^{1,0}\cdot 0 = \Lambda_2^{2,0}.\]
Furthermore, using that $\Lambda\in E_\varepsilon$ we deduce $\Lambda_2^{1,0}=0$ since
\[|\Lambda_2^{1,0}|^2 = \varepsilon(\Lambda_3^{0,1}-|\Lambda_1^{1,0}|^2)=0.\]
 
Let $\sigma'\in {\rm Aut}_0(\mathbb H^3_\varepsilon)$, corresponding to $\gamma' = (\lambda',r',U',c')\in \Gamma'_\varepsilon$, be such that $j_0^4(\sigma' \circ \Lambda) = \Lambda$. We will follow the same computations as in Lemma \ref{proact}. Looking at the third component of the first jet, by (\ref{param}) we must have $(\lambda')^2 = 1$, hence $\lambda' = 1$ since $\lambda'\in \mathbb R^+$.

Looking now at the first two components of the first jet, we get
\begin{align*}
U' \left ( \begin{array}{c} 1  \\ 0 \end{array}\right) & = \left ( \begin{array}{c} 1  \\ 0 \end{array}\right), \ {i.e.}\  U' = \left( \begin{array}{cc} 1 & 0 \\ 0 & e^{i\theta} \end{array} \right),\\
U' \left( \begin{array}{c} 0  \\ \Lambda_2^{0,1} \end{array} \right) +  c' & = \left( \begin{array}{c} 0  \\ \Lambda_2^{0,1} \end{array} \right),\  {i.e.} \ c_1' = 0, c_2' = \Lambda_2^{0,1}(e^{i\theta}-1).
\end{align*}
Using the equations above, the ones for the second jet become as follows:
\begin{align*}
e^{i\theta}\Lambda_2^{2,0} & =\Lambda_2^{2,0}, \  {i.e.}\ \theta = 0, c_2' = 0, \\
-r' & = \Lambda_3^{0,2} = 0,
\end{align*}
which shows that $\sigma'=\id$. 
\end{proof}

\section{A linear criterion for local rigidity}\label{sec:linear}

We are now going to refer to the notation of Theorem \ref{jetparam}. Let $A\subset  J_0^4$ be the real-analytic set defined as
\begin{equation}\label{defas}
A = \{\Lambda\in  J_0^4 \colon q(\Lambda, \bar \Lambda)\neq 0,\, c_j(\Lambda,\overline \Lambda) = 0, j\in \mathbb N\}.
\end{equation}
By Theorem \ref{jetparam}, $A$ coincides with the set of the $4$-jets of the transversal, $2$-nondegenerate local embeddings of $M$ into $M'$, i.e.\ $A = j_0^4(\mathcal F_2)$, which implies that $A \subset E_\varepsilon$ if $M'$ is Levi-nondegenerate by Remark \ref{furthermore}.

The restriction of $\Phi$ to $\mathcal U \cap (\mathbb C^2 \times A)$ gives rise to a map 
\[A \ni \Lambda \to \Phi(\Lambda) \in (\mathbb C\{z,w\})^3, \ \Phi(\Lambda) (z,w) = \Phi(z,w,\Lambda)\]
from $A$ to the space $(\mathbb C\{z,w\})^3$. By Theorem \ref{jetparam}, the image of $A$ under $\Phi$ is actually $\mathcal F_2$. We will need the following remark:
\begin{lemma}\label{Ginv}
The analytic set $A$ is invariant under the action of $G$. Moreover, the map $\Phi:A\to\mathcal F_2$ is an equivariant homeomorphism. 
\end{lemma}
\begin{proof}
From Theorem \ref{jetparam} follows that $\Phi:A\to \mathcal F_2$ is one to one, since the projection $j_0^4$ is the inverse of $\Phi$. The continuity of the map $\Phi$ follows directly from the fact that $\Phi(z,w,\Lambda): \mathcal U\to \mathbb C^3$ is holomorphic (which is a much stronger statement). Since $j_0^4:(\mathbb C\{z,w\})^3\to J_0^4$ is also continuous, $\Phi$ admits a continuous inverse and thus it is a homeomorphism $A\to \mathcal F_2$.

Let now $\Lambda=j_0^4(\Phi(\Lambda) )\in A$ with $\Phi(\Lambda) \in \mathcal F_2$, and let $\Lambda' = j_0^4 (\sigma'\circ \Lambda \circ \sigma^{-1})$ (where $(\sigma,\sigma')\in G$). Then, as noticed in section \ref{isogro}, we also have $\sigma'\circ \Phi(\Lambda)  \circ \sigma^{-1} \in \mathcal F_2$ and thus $\Lambda' = j_0^4(\sigma'\circ \Phi (\Lambda) \circ \sigma^{-1}) \in j_0^4(\mathcal F_2) = A$. This shows that $A$ is $G$-invariant and that $\Phi$ is equivariant. 
\end{proof}

 Let $N\subset A$ be any regular (real-analytic) submanifold, and fix $\Lambda_0 \in N$. In what follows we focus on $\Phi|_N$. 
 By Theorem \ref{jetparam}, the image of $N$ through $\Phi$ is $(j_0^4)^{-1}(N)\cap \mathcal F_2$. 
 Note that if we restrict to a small enough neighborhood $N'$ of $\Lambda_0$ in $N$, 
 the maps $\Phi(\Lambda) $ for all $\Lambda\in N'$ all have a common radius of convergence $R$, 
 so that we can consider the restriction of $\Phi$ to $N'$ as a map 
 valued in  the Banach space ${\rm Hol}(\overline{B_R(0)}, \mathbb C^3)$.

We also remark that the map $\Phi: N \to (\mathbb C\{z,w\})^3$ is of class $C^\infty$. Its Fr\'echet derivative at $\Lambda_0$ is the map
\[D\Phi(\Lambda_0): T_{\Lambda_0}N \to T_{\Phi(\Lambda_0)}(\mathbb C\{z,w\})^3\cong (\mathbb C\{z,w\})^3,\ \]
\[ T_{\Lambda_0}N \ni  \Lambda' \to \frac{\partial \Phi}{\partial \Lambda '}(z,w,\Lambda_0)=
\frac{d}{dt}\biggr|_{t=0} \Phi(z,w,\Lambda_0 + t \Lambda')\in (\mathbb C\{z,w\})^3. \]
 
Write the components of  $\Phi(\Lambda_0)$ as $(\Phi_1, \Phi_2, \Phi_3)$. We need to identify a particular subspace of the tangent space of $(\mathbb C\{z,w\})^3$:
\begin{definition}\label{def:infdef} Assume that $M\subset\mathbb C^2$, $M'\subset\mathbb C^3$ are hypersurfaces through $0$, and 
that
$H\colon (\mathbb C^2,0) \to (\mathbb C^3,0)$ is a map with $H(M) \subset M'$. Then 
we say that a vector 
\[V =  \alpha(z,w)\frac{\partial}{\partial z_1'} + \beta(z,w)\frac{\partial}{\partial z_2'} + \gamma(z,w)\frac{\partial}{\partial w'} \in T_{H}(\mathbb C\{z,w\})^3\] 
is an \emph{infinitesimal deformation of $H$} if the real part of $V$ is tangent to $M'$ along $H(M)$, i.e.\ if 
for any defining function $\rho$ of $M'$
\begin{equation}\label{infdef}
{\rm Re}\left(\alpha(z,w)\frac{\partial \rho}{\partial z_1'}(H,\overline H) + \beta(z,w)\frac{\partial \rho}{\partial z_2'}(H,\overline H) + \gamma(z,w)\frac{\partial \rho}{\partial w'}(H,\overline H)\right)=0 \ \ {\rm for} \ (z,w)\in M. 
\end{equation}
We denote this subspace of $T_{H}(\mathbb C\{z,w\})^3$ by $\mathfrak {hol}_0 (H)$.

In the case when $M'=\mathbb H^3$, we can write the previous condition as
\begin{equation}\label{infdef2}
{\rm Re} \left(2 \alpha(z,w)\overline{H_1(z,w)} + 2 \beta(z,w)\overline{H_2(z,w)} - i\gamma(z,w) \right ) = 0 \ \ {\rm for} \ (z,w)\in M. 
\end{equation}

\end{definition}
The motivation for the definition above is the following observation:
\begin{lemma}\label{contain}
The image of $T_{\Lambda_0}N$ by $D\Phi(\Lambda_0)$ is contained in $\mathfrak {hol}_0 (\Phi(\Lambda_0))$.
\end{lemma}
\begin{proof}
Given $\Lambda'\in T_{\Lambda_0}(N)$, let $\Lambda(t)$ ($t\in \mathbb R$) be any smooth curve contained in $N$ such that $\Lambda(0) = \Lambda_0 $ and $\frac{\partial \Lambda}{\partial t}(0) = \Lambda' $. For all $t\in \mathbb R$, write $(\hat \Phi_1(z,w,t),\hat  \Phi_2(z,w,t),\hat \Phi_3(z,w,t))$ for the components of $\Phi(\Lambda(t))$. Then by definition
\[D\Phi(\Lambda_0)[\Lambda'] = \frac{\partial\hat  \Phi_1}{\partial t}(z,w,0) \frac{\partial}{\partial z_1'} + \frac{\partial\hat  \Phi_2}{\partial t}(z,w,0) \frac{\partial}{\partial z_2'} + \frac{\partial\hat  \Phi_3}{\partial t}(z,w,0) \frac{\partial}{\partial w'}.   \]
Since $\Phi(\Lambda(t))$ maps $M$ into $M'$, we have 
\[\rho(\Phi(z,w,\Lambda(t)),\overline{\Phi(z,w,\Lambda(t))} )= 0\] 
for all $t\in \mathbb R$ and $(z,w)\in M$.
Differentiating with respect to $t$ and computing at $t=0$ we get
\[ {\rm Re}\left(\frac{\partial\hat  \Phi_1}{\partial t}(z,w,0)\frac{\partial \rho}{\partial z_1'}(\Phi,\overline \Phi) + \frac{\partial\hat  \Phi_2}{\partial t}(z,w,0)\frac{\partial \rho}{\partial z_2'}(\Phi,\overline \Phi) + \frac{\partial\hat  \Phi_3}{\partial t}(z,w,0)\frac{\partial \rho}{\partial w'}(\Phi,\overline \Phi)\right)=0\]
for $(z,w)\in M$, which is equivalent to (\ref{infdef}).
\end{proof}
\begin{remark}\label{trivsol} If $H\colon(\C^2,0 ) \to (\C^3,0)$ maps $M$ into $M'$, 
the restriction to  $H(\mathbb C^2)$ of 
any infinitesimal automorphism $Z\in \mathfrak{hol}_0(M')$ of $M'$ solves (\ref{infdef}), 
since in fact the real part of $Z$ is tangent to $M'$ everywhere and not only along $H(\mathbb C^2)$. To be
more exact, if $\mathfrak{hol}_0(M')|_{H(M)}$ denotes the space of infinitesimal automorphisms
of $M'$ restricted to $H(M)$, then 
$\mathfrak{hol}_0(M')|_{H(M)} \subset \mathfrak{hol}_0 (H)$. Thus, if we write
\[ \mathfrak{hol}_{H(M)} (M') =  \{X \in \mathfrak{hol}_0(M') \colon X|H(M) = 0\}\] it follows that
\[ \mathfrak{hol}_0(M')|_{H(M)} = \faktor{\mathfrak{hol}_0(M')}{ \mathfrak{hol}_{H(M)} (M') },   \]  
so that  
 $\dim \mathfrak {hol}_0 (H) \geq \dim \mathfrak {hol}_0 (M') - \dim \mathfrak{hol}_{H(M)} (M')$.

 Moreover, if   $G'=  \{\id\}\times {\rm{Aut}}_0 (M')$ acts freely on $H$, then $ \mathfrak{hol}_{H(M)} (M') = \{0\}$. 
Thus, if $M'=\mathbb H^3$, the dimension of $\mathfrak {hol}_0 (H)$ is always at least $10$. For instance, if $H = \widetilde H$ with $\widetilde H(z,w)=(z,F(z,w),w)$ as given in Lemma \ref{coord}, this \lq\lq trivial\rq\rq\ subspace of solutions of (\ref{infdef}) can be parametrized by
\[ \alpha(z,w)= (t+ih_{11})  z + h_{12}F(z,w) + i\frac{\overline b_1}{2}w +  b_1 z^2 + b_2 z F(z,w)  + s z w, \]
\[ \beta(z,w) = (t+ih_{22}) F(z,w) - \overline h_{12}z +  i\frac{\overline b_2}{2}w + b_1 z F(z,w) + b_2 {F(z,w)}^2 + s wF(z,w), \]
\[ \gamma(z,w) =   2tw + b_1 z w+ b_2 w F(z,w)  + s w^2,\]
where $(t,h_{11},h_{22},s,h_{12},b_1,b_2) \in \mathbb R^4\times \mathbb C^3$.
\end{remark}
In the following lemma, we use some structural results concerning the solutions of a linear equation of the kind  (\ref{infdef}); we defer their treatment to section \ref{linearizedproblem}.
\begin{lemma} \label{dim10}
 Let $\Lambda_0\in A$,  and suppose that
  $\dim_{\mathbb R} \mathfrak {hol}_0 (\Phi(\Lambda_0)) = \ell$.
 Then there exists a neighborhood $U$ of $\Lambda_0$ in $J_0^4$ such that, 
 if $N\subset A$ is a submanifold such that $N\cap U\neq \emptyset$, then the real dimension of $N$ is at most $\ell$.
\end{lemma}
\begin{proof}
First, we note that $\dim_{\mathbb R}\mathfrak {hol}_0 (\Phi(\Lambda) ) \leq \ell$ for all $\Lambda$ in a neighborhood of $\Lambda_0$ in $A$: 
this is a consequence of Corollary \ref{semicont} and of the continuity of the map $\Phi:A \to (\mathbb C\{z,w\})^3$.

As noted above, if we shrink the neighborhood of $\Lambda_0$ we can assume that all the $\Phi(\Lambda) $ have a common radius of convergence $R$, so we can assume that $N$ is a submanifold of this smaller neighborhood and that $\Phi$ actually maps $N$ into the Banach space ${\rm Hol}(\overline{B_R(0)}, \mathbb C^3)$, and that 
$\dim_{\mathbb R}\mathfrak {hol}_0 (\Phi(\Lambda) ) \leq \ell$ for all $\Lambda\in N$. 
Also observe that by  Theorem \ref{jetparam} the $4$-th jet of $\Phi(\Lambda) $ coincides with $\Lambda$ for any $\Lambda\in N$: in particular, the map $\Phi :N \to {\rm Hol}(\overline{B_R(0)}, \mathbb C^3)$ is injective.


Suppose now that the minimum of the dimension of the kernel of $D\Phi(\Lambda) : T_{\Lambda}N \to T_{\Phi(\Lambda) }(\mathbb C\{z,w\})^3$ for $\Lambda\in N$ is at least $1$. Then by the rank theorem (see for example \cite[Theorem 6.3.34]{BER2} for a version valid in the setting of Banach manifolds) there exists a submanifold $N'$ of $N$, of positive dimension, such that $\Phi(\Lambda_1) = \Phi(\Lambda_2) $ for all $\Lambda_1,\Lambda_2\in N'$, which contradicts the injectivity of $\Phi$. 

It follows that the linear map $D\Phi(\Lambda) $ is injective for some $\Lambda\in N$ (indeed, on a dense set), which implies $\dim N = \dim D\Phi(\Lambda) (T_\Lambda(N))$. From Lemma \ref{contain} we have that $D\Phi(\Lambda) (T_\Lambda(N)) \subset \mathfrak {hol}_0 (\Phi(\Lambda) )$, hence $\dim N \leq \ell$.
\end{proof}
The main goal of this section are the following results, providing sufficient conditions for local rigidity which depend on the space of infinitesimal deformations defined above: this allows to treat the local rigidity problem in a linear way. We first state a result which is valid for a general target manifold $M'\subset\mathbb C^3$:

\begin{theorem}\label{suffcon1}
Let $M,M'$ be as in Theorem \ref{jetparam}, let $A$ be defined as in (\ref{defas}), and let $\Lambda_0\in A$ be such that $\dim_{\mathbb R} \mathfrak {hol}_0 (\Phi(\Lambda_0)) = 0$. Then $\Phi(\Lambda_0)$ is locally rigid. \end{theorem}
\begin{proof}
By Lemma \ref{dim10}, there is a neighborhood $U$ of $\Lambda_0$ in $J_0^4$ such that $U\cap A$ does not contain any manifold of positive 
dimension: it follows that $U\cap A$ is a discrete set. By Remark~\ref{rem:equcon}, $\Phi(\Lambda_0)$ is locally rigid.
\end{proof}

In the next result we look more closely at the case when $M'$ is Levi-nondegenerate, and we relax the assumption $\dim_{\mathbb R} \mathfrak {hol}_0 (\Phi(\Lambda_0)) = 0$. We use the same scheme of the proof of Theorem \ref{suffcon1}, but the presence of a positive dimensional isotropy group ${\rm Aut}_0(M')$ (even non-compact if $M'\cong \mathbb H^3_{\varepsilon}$) means that some additional care is required.

\begin{theorem}\label{suffcon2}
Let $M,M'$ be as in Theorem \ref{jetparam} with $M'$ Levi-nondegenerate, let $A$ be defined as in (\ref{defas}), and let $\Lambda_0\in A$ be such that $\dim_{\mathbb R} \mathfrak {hol}_0 (\Phi(\Lambda_0)) = \dim_{\mathbb R}\mathfrak{hol}_0(M') = \ell$. Then $\Phi(\Lambda_0)$ is locally rigid.
\end{theorem}
\begin{proof}
By Lemma \ref{coord2} we can choose coordinates in $\mathbb C^2$ and $\mathbb C^3$ 
such that $\Phi(\Lambda_0)(z,w) = \widetilde H(z,w) = (z,F(z,w),w)$ and ${\rm Aut}_0(M')$ 
is a subgroup of ${\rm Aut}_0(\mathbb H^3_\varepsilon)$. Let $G'=\{\id\}\times {\rm Aut}_0(M')$, and denote by $G'_{\id}$ the connected component of the identity in $G'$. By Lemma \ref{free}, the action 
of $G'_{\id}$ is free on a neighborhood of $\Lambda_0$ 
in $E_\varepsilon$, hence also on a $G'_{\id}$-invariant neighborhood. By Lemmas \ref{proact} 
and \ref{proact2} this action is also proper. 

We can thus apply the real-analytic version of the local slice theorem for free and proper actions (see \cite{DK}): put $m = \dim_{\mathbb R} E_\varepsilon - \ell$. There exist
\begin{itemize}
\item a (germ of) $m$-dimensional real-analytic submanifold $S$ ($\Lambda_0\in S$) of $E_\varepsilon$, transversal to the orbits of $G'_{\id}$, called the \emph{local slice},
\item a $G'_{\id}$-invariant neighborhood $\mathcal V$ of $\Lambda_0$ in $E_\varepsilon$, containing $S$,
\item a real-analytic $G'_{\id}$-equivariant diffeomorphism $\varphi : B^m\times G'_{\id} \to \mathcal V$, where $B^m\subset \mathbb R^m$ is the real $m$-dimensional ball,
\end{itemize}
such that $\varphi(0,\id) = \Lambda_0$ and $\varphi|_{B^m\times \{\id\}}$ induces a diffeomorphism between $B^m$ and $S$.

Let $A'=A\cap S$: then $A'$ is  a real-analytic subset of $\mathcal V$. Moreover $A'$ is $0$-dimensional. 
Indeed, otherwise it must contain a real $1$-dimensional curve $\gamma$, 
embedded in (an open subdomain of) $\mathcal V$. Then $\gamma'=\varphi^{-1}(\gamma)$ 
is a $1$-dimensional curve contained in $B^m\times \{\id\}$, 
and $N'=\{(x,g')\in B^m\times G'_{\id}: x\in \gamma' \}$ is an $(\ell+1)$-dimensional 
submanifold of $B^m\times G'_{\id}$. Defining $N= \varphi(N')$, we 
would have that $N$ is $(\ell+1)$-dimensional and is contained in
 $A$ (because of the $G'_{\id}$-invariance of $A$, see Lemma \ref{Ginv}), contradicting Lemma \ref{dim10}.

Since $A'$ is a real-analytic set of dimension $0$, it is a discrete subset of $\mathcal V$, which we can assume to reduce to only $\Lambda_0$ up to shrinking $\mathcal V$ and $S$.  It follows from the slice theorem that $A\cap \mathcal V$ consists of the orbit $G'_{\id}\cdot \Lambda_0$ of $G'_{\id}$ through $\Lambda_0$. Let $G_{\id}$ be the connected component of the identity in $G$: then the orbit $G'_{\id}\cdot \Lambda_0$ must also coincide with the orbit $G_{\id}\cdot \Lambda_0$ of $G_{\id}$ through $\Lambda_0$. Indeed, the latter is also a (connected) submanifold by the properness of the $G$-action, since when the action is proper the orbit is diffeomorphic to $G_{\id}/{\rm stab}(\Lambda_0)$, where ${\rm stab}(\Lambda_0)$ is the compact stabilizer of $\Lambda_0$, and moreover it contains $G'_{\id}\cdot \Lambda_0$ but by Lemma \ref{dim10} it cannot be of higher dimension.

In summary, the arguments above show that there exists a neighborhood $V$ of $\Lambda_0$ such that $A\cap V = (G\cdot \Lambda_0)\cap V$: since the map $\Lambda \mapsto \Phi(\Lambda) $ is an equivariant homeomorphism between $A$ and $\mathcal F_2$ (see Lemma \ref{Ginv}), there is also a neighborhood $U$ of $\Phi (\Lambda_0)$ such that $\mathcal F_2\cap U = (G\cdot \Phi (\Lambda_0))\cap U$. By Remark~\ref{rem:equcon}, $\Phi (\Lambda_0)$ is locally rigid.
\end{proof}

\section{Solving the linearized problem}\label{linearizedproblem}
In this section our aim is to understand the properties of the space of solutions of a linear equation of the kind 
(\ref{infdef}). The arguments work in a more general setting than the one considered in the previous sections, thus 
we will consider a generic minimal real-analytic  CR submanifold $M\subset \mathbb C^N$ of 
CR dimension $n$ and real codimension $d$ (so that $N=n+d$), $0\in M$, 
in normal coordinates around $0$ (not necessarily the ones in 
which $M$ is written as in (\ref{oftheform})). Choosing coordinates
 $(z,w)\in \mathbb C^n_z\times \mathbb C^d_w=\mathbb C^N$, this means 
 (cf. \cite{BER2}) that the complexification $\mathcal M\subset \mathbb C^{2N}_{z,\chi,w,\tau}$ of $M$ is 
 given by the equation
\[w = Q(z,\chi, \tau) \ \ {\rm (or\ equivalently,} \ \tau = \overline Q(\chi,z, w) {\rm )}\]
for a suitable germ of holomorphic map $Q:\mathbb C^{2n+d}\to\mathbb C^{d}$ satisfying the properties
\[Q(z,0,\tau) \equiv Q(0,\chi,\tau) \equiv \tau, \ \ \ \ Q(z,\chi, \overline Q (\chi, z, w)) \equiv w.\]
In these coordinates, the (respectively CR and anti-CR) vector fields tangent to $\mathcal M$ are given by
\[L_j = \frac{\partial}{\partial \chi_j} + \sum_{k=1}^d\overline Q^k_{\chi_j}(\chi,z,w) \frac{\partial}{\partial \tau_k }, \ \ \overline L_j = \frac{\partial}{\partial z_j} + \sum_{k=1}^d Q^k_{z_j}(z,\chi,\tau) \frac{\partial}{\partial w_k }, \]
 for $1\leq j \leq n$. It will also be convenient to consider the following vector fields, which are neither CR nor anti-CR but are nevertheless tangent to $\mathcal M$:
\[T_\ell = \frac{\partial}{\partial w_\ell} + \sum_{k=1}^d \overline Q^k_{w_\ell}(\chi,z,w) \frac{\partial}{\partial \tau_k}, \ \ \ \ S_j = \frac{\partial}{\partial z_j} + \sum_{k=1}^d\overline Q^k_{z_j}(\chi,z,w) \frac{\partial}{\partial \tau_k} \]
where $1\leq \ell \leq d$, $1\leq j\leq n$. 

Let $R\subset  \mathbb C\{Z',\zeta'\}^{N'}$ be a real subspace. We will say that a holomorphic map
 $H=(H_1,H_2,\ldots,H_{N'})\in (\mathbb C\{z,w\})^{N'}$
  is {\em $\kappa$-nondegenerate} with respect 
  to $R$ if  the following condition is 
satisfied: $\kappa$ is the smallest integer such that there exists a sequence
 $(\iota_1,\ldots,\iota_{N'})$ of 
 multiindices $\iota_\ell\in\mathbb N^n$ such that $0\leq|\iota_\ell|\leq \kappa$ for  $ \ell = 1, \dots,  N'$, and $r^1 , \dots, r^{N'} \in R$ such that, for 
 $r^j = (r^j_1,\dots , r^j_{N'} )$ 
\begin{equation} \label{folcon}
s = \det \left(\begin{array}{ccc} 
 L^{\iota_1}r^1_1(H(z,w),\overline H(\chi,\tau)) & \cdots & L^{\iota_1}r^1_{N'}(H(z,w),\overline H(\chi,\tau)) \\  \vdots & \ddots & \vdots \\
 L^{\iota_{N'}} r^{N'}_1(H(z,w),\overline H(\chi,\tau)) & \cdots & L^{\iota_{N'}} r_{N'}^{N'}(H(z,w),\overline H(\chi,\tau))\end{array}\right),
\end{equation}
we have $s(0)\neq 0$. If $N=2$, $N'=3$, $\kappa=2$, and $R$ is
generated by  $(\rho_{z'_1},\rho_{z'_2}, \rho_{w'})$ for a certain real defining function $\rho\in \mathbb C\{z,w,\chi,\tau\}$ 
 and $H$ is an embedding of $M$ into $\{\rho=0\}$, the condition above  amounts to $H$ being $2$-nondegenerate. However, in this section we will not need to make these assumption.

We are going to consider a generalization of (\ref{infdef}). More precisely, for
 fixed $R$, we want to study the space of the
  $(\alpha_1(z,w), \ldots, \alpha_{N'}(z,w))\in (\mathbb C\{z,w\})^{N'}$,
   $\alpha_j(0,0)=0$, which solve the linear equations
\begin{multline}
\label{lineq}
\sum_{j=1}^{N'} r_j(H(z,w),\overline H(\chi,\tau))\alpha_j(z,w) + \overline r_j(\overline H(\chi,\tau), H(z,w))\overline \alpha_j(\chi,\tau) = 0 \\ {\rm for}\  w = Q(z,\chi,\tau), \quad r=(r_1,\dots,r_{N'}) \in R,
\end{multline}
and how this space depends on the data $(H, Q)\in (\mathbb C\{z,w\})^{N'}\times (\mathbb C\{z,\chi,\tau\})^d$. We will approach this problem with the techniques of reflection which are normally used to obtain jet parametrization results for holomorphic maps.

\subsection{Reflection identity}
The first step is to differentiate (\ref{lineq}) along $L$ enough times, in order to obtain the linear system

\[ 
\sum_{j=1}^{N'} L^{\iota_k}r^k_j(H,\overline H)\alpha_j(z,w) = - L^{\iota_k}\sum_{j=1}^{N'}\overline r^k_j(\overline H, H)\overline \alpha_j(\chi,\tau),\quad k = 1,\dots,N' 
 \]
for $w = Q(z,\chi,\tau)$, and where $r^1,\dots,r^{N'}$ are chosen 
such that \eqref{folcon} is satisfied. We solve this linear system for $\alpha_1(z,w),\ldots, \alpha_{N'}(z,w)$ as a rational function of $L^\iota \overline r^j(\overline H, H)$, $L^\iota r^j(H,\overline H)$, $L^\iota \overline \alpha_\ell(\chi,\tau)$  which is in fact linear in the last entries, and non-singular at $0$. This is possible because of (\ref{folcon}). More precisely, we get the following (here and in the subsequent lemmas $j_1\in\mathbb N^n$ and $j_2\in\mathbb N^d$ are multiindices):

\begin{lemma}\label{univpol}
There are universal polynomial maps $p^{h,j_1,j_2}_\ell$ (where $0\leq |j_1|+|j_2|\leq \kappa$, $1\leq \ell,h \leq N'$) satisfying the following property. For any $\mathcal M= \{w = Q(z,\chi,\tau)\}$, $R$, $H$ as above (with $H$ fulfilling (\ref{folcon})), any solution $(\alpha_1, \ldots, \alpha_{N'})$ of the equation (\ref{lineq})  satisfies the identity 
\begin{equation}\label{reflid}
\alpha_\ell(z,w) = \frac{1}{s} \sum_{\substack{0\leq |j_1| + |j_2| \leq \kappa\\ 1\leq h \leq N'}} p_\ell^{h,j_1,j_2}(\partial^\kappa r^k, \partial^\kappa \overline r^k, \partial^\kappa \overline Q)  \frac {\partial^{|j_1|+|j_2|} \overline \alpha_h}{\partial \chi^{j_1}\partial \tau^{j_2} }(\chi,\tau), \  w = Q(z,\chi,\tau),
\end{equation}
for all $1\leq \ell \leq N'$, where the symbols $\partial^{\kappa} r^k, \partial^{\kappa} \overline r^k, \partial^\kappa \overline Q$ ($1\leq k \leq N'$) represent the set of all derivatives  of order at most $\kappa$ of $r^k(H(z,w), \overline H(\chi,\tau)), \overline r^k(\overline H(\chi,\tau), H(z,w))$, $\overline Q(\chi, z, w)$, and $s$ is given by (\ref{folcon}).
\end{lemma}

By repeatedly differentiating (\ref{reflid}) along the fields $T$, $S$, and noting that $S^{j_1}T^{j_2}\alpha_\ell(z,w) = \alpha_{\ell, z^{j_1}w^{j_2}}(z,w)$ for all multiindices $j_1\in \mathbb N^n$, $j_2\in \mathbb N^d$, we immediately obtain the following:

\begin{cor}
Let $m\in \mathbb N$.
There are universal polynomial maps $p^{h,j_1,j_2}_{\ell,n_1,n_2}$,
where $0\leq |j_1|+|j_2|\leq m + \kappa$, $0\leq |n_1| + |n_2| \leq m$, $1\leq \ell,h \leq N'$, satisfying the following property. For any $\mathcal M= \{w = Q(z,\chi,\tau)\}$, $R$, $H$ as above (with $H$ fulfilling (\ref{folcon})), any solution $(\alpha_1, \ldots, \alpha_{N'})$ of the equation (\ref{lineq})  satisfies the identity 
\begin{equation}\label{diffreflid}
\frac{\partial^{|n_1| + |n_2|} \alpha_\ell}{\partial z^{n_1}\partial w^{n_2}}(z,w) = \frac{1}{s^{m+1}} \sum_{\substack{0\leq |j_1| + |j_2| \leq m + \kappa\\ 1\leq h \leq N'}} p_{\ell,n_1,n_2}^{h,j_1,j_2}(\partial^{m+\kappa} r^k, \partial^{m+\kappa} \overline r^k, \partial^{m+\kappa} \overline Q)  \frac {\partial^{|j_1|+|j_2|} \overline \alpha_h}{\partial \chi^{j_1}\partial \tau^{j_2} }(\chi,\tau), 
\end{equation}
$w = Q(z,\chi,\tau)$, for all $1\leq \ell \leq N'$, where the symbols $\partial^{m+\kappa} r^k, \partial^{m+\kappa}\overline r^k, \partial^{m+\kappa} \overline Q$ ($1\leq k \leq N'$) represent the set of all derivatives of $r^k(H(z,w), \overline H(\chi,\tau)), \overline r^k(\overline H(\chi,\tau), H(z,w)), \overline Q(\chi,z,w)$ of order at most $m+\kappa$, and $s$ is as in (\ref{folcon}).
\end{cor}

 \subsection{Iteration along the Segre sets} The next step is to evaluate (\ref{reflid}) along certain subvarieties of $\mathbb C^N$, called \emph{Segre sets}. In order to do so we need to introduce some notation. For any $j\in \mathbb N$ let $(x_1,\ldots, x_j)$ ($x_\ell\in \mathbb C^n$) be coordinates for $\mathbb C^{nj}$. The \emph{Segre map} of order $q\in \mathbb N$ is the map $S^q_0:\mathbb C^{nq}\to \mathbb C^N$ inductively defined as follows:
 \[S^1_0(x_1) = (x_1,0), \ \ S^q_0(x_1,\ldots,x_q) = \left (x_1, Q\left(x_1,\overline S^{q-1}_0(x_2,\ldots,x_q) \right)\right)\]
where we denote by $\overline S^{q-1}_0$ the power series whose coefficients are conjugate to the ones of $S^{q-1}_0$. The $q$-th Segre set $\mathcal S^q_0\subset \mathbb C^N$ is then the image of the map $S^q_0$: we have for instance
 \[\mathcal S^1_0 = \{(z,0)\in \mathbb C^N: z\in \mathbb C^n\},\]
 \[\mathcal S^2_0 = \{(z, Q(z,\chi,0))\in \mathbb C^N: z,\chi\in \mathbb C^n \}. \]
In what follows we will use the notation $x^{[j;k]} = (x_j,\ldots,x_k)$. Fixed $q\in \mathbb N$, we begin by putting $z=x_1$, $\chi=x_2$, $\tau=\overline Q(x_2,  S^{q-2}_0(x^{[3;q]}))$ -- and hence $w = Q(x_1,\overline S^{q-1}_0(x^{[2;q]}))$ -- in the identity (\ref{reflid}), in order to obtain
\begin{equation}\label{firsteval}
\alpha_\ell(S^q_0(x^{[1;q]})) = \frac{1}{s} \sum_{\substack{0\leq |j_1| + |j_2| \leq \kappa\\ 1\leq h \leq N'}} p_\ell^{h,j_1,j_2}(\partial^\kappa r^k, \partial^\kappa \overline r^k, \partial^\kappa \overline Q)  \frac {\partial^{j_1+j_2} \overline \alpha_h}{\partial \chi^{j_1}\partial \tau^{j_2} }(\overline S^{q-1}_0(x^{[2;q]}))
\end{equation}
where all the functions $\partial^\kappa r^k(H,\overline H), \partial^\kappa \overline r^k(\overline H,H), \partial^\kappa \overline Q$ are evaluated at $z=x_1$, $\chi=x_2$, $\tau=\overline Q(x_2,  S^{q-2}_0(x^{[3;q]}))$, $w = Q(x_1,\overline S^{q-1}_0(x^{[2;q]}))$. This equation means that one can determine the value of any solution of (\ref{lineq}), at least along $\mathcal S^q_0$, by knowing the values of its derivatives along $\mathcal S^{q-1}_0$. To determine the latter, we put $\chi=x_2$, $z=x_3$, $w=Q(x_3,  \overline S^{q-3}_0(x^{[4;q]}))$ -- and hence $\tau = \overline Q(x_2, S^{q-2}_0(x^{[3;q]}))$ -- in the conjugate of (\ref{diffreflid}):
\begin{equation}\label{secondeval}
\frac{\partial^{|n_1| + |n_2|} \overline \alpha_\ell}{\partial \chi^{n_1}\partial \tau^{n_2}}(\overline S^{q-1}_0(x^{[2;q]})) =\end{equation}
\[= \frac{1}{\overline s^{m+1}} \sum_{\substack{0\leq |j_1| + |j_2| \leq m + \kappa\\ 1\leq h \leq N'}} \overline p_{\ell,n_1,n_2}^{h,j_1,j_2}(\partial^{m+\kappa} r^k, \partial^{m+\kappa} \overline r^k, \partial^{m+\kappa} \overline Q)  \frac {\partial^{|j_1|+|j_2|} \alpha_h}{\partial z^{j_1}\partial w^{j_2} }(S^{q-2}_0(x^{[3;q]})), \]
where the functions $\partial^{m+\kappa} r^k, \partial^{m+\kappa} \overline r^k, \partial^{m+\kappa} \overline Q$ are evaluated at $\chi=x_2$, $z=x_3$, $w=Q(x_3,  \overline S^{q-3}_0(x^{[4;q]}))$, $\tau = \overline Q(x_2, S^{q-2}_0(x^{[3;q]}))$. 
By substituting (\ref{secondeval}) for $m=\kappa$ into (\ref{firsteval}), we get that the values of the $\alpha_\ell$ along $\mathcal S^q_0$ are determined by the values of their $2\kappa$-th order jet along $\mathcal S^{q-2}_0$. Iterating this argument $q$ times (note that in the last step it is enough to put $\tau=0$, $z=x_q$, $\chi=w=0$ in (\ref{diffreflid}) to express the derivatives of $\alpha_\ell$ along $\mathcal S^1_0$ in terms of the derivatives at $(0,0)$) we prove the following:

\begin{lemma}\label{secseg}
Fix $q\in \mathbb N$. There are universal polynomial maps $q^{h,j_1,j_2}_\ell$, where $0\leq |j_1|+|j_2|\leq q\kappa$, $1\leq \ell,h \leq N'$, satisfying the following property. For any $\mathcal M= \{w = Q(z,\chi,\tau)\}$, $R$, $H$ as above (with $H$ fulfilling (\ref{folcon})), any solution $(\alpha_1, \ldots, \alpha_{N'})$ of the equation (\ref{lineq})  satisfies the identity
\begin{equation}\label{secondsegre}
\alpha_\ell(S^q_0(x^{[1;q]})) = 
\begin{cases}\displaystyle
 \frac{1}{S}  \sum_{\substack{0\leq |j_1| + |j_2| \leq q\kappa\\ 1\leq h \leq N'}} q_\ell^{h,j_1,j_2}(\partial^{q\kappa} r^k, \partial^{q\kappa} \overline r^k, \partial^{q\kappa} \overline Q, \partial^{q\kappa} Q)  \frac {\partial^{|j_1|+|j_2|} \alpha_h}{\partial z^{j_1}\partial w^{j_2} }(0,0) &q  \text{ even}, \\
 \displaystyle
 \frac{1}{S} \sum_{\substack{0\leq |j_1| + |j_2| \leq q\kappa\\ 1\leq h \leq N'}} q_\ell^{h,j_1,j_2}(\partial^{q\kappa} r^k, \partial^{q\kappa} \overline r^k, \partial^{q\kappa} \overline Q, \partial^{q\kappa} Q)  \frac {\partial^{|j_1|+|j_2|} \bar \alpha_h}{\partial \chi^{j_1}\partial \tau^{j_2} }(0,0) &q  \text{ odd}. 
 \end{cases}
\end{equation}
In the expression (\ref{secondsegre}), $S$ is the product of $q ( 1 + \frac{\kappa(q-1)}{2})$ factors, each one of which is equal to either $s$ or $\overline s$ (with $s$ as in (\ref{folcon})), evaluated either at $(0,0)$ or along $(z,w,\chi,\tau) = (S^{q-j}(x^{[j+1;q]}), \overline S^{q-j-1}(x^{[j+2;q]}))$ for some $0\leq j\leq q-1$. Furthermore, the symbols $\partial^{q\kappa} r^k, \partial^{q\kappa} \overline r^k$, $\partial^{q\kappa} \overline Q, \partial^{q\kappa} Q$ represent all the partial derivatives of $r,\overline r,Q,\overline Q$ of order less or equal to $q\kappa$, evaluated along $(H\circ S^{q-j}(x^{[j+1;q]}),\overline H \circ  \overline S^{q-j-1}(x^{[j+2;q]}) )$
for some $0\leq j\leq q-1$.
\end{lemma}


\subsection{Jet parametrization} In this section we will use the notation
 $J_0^k = \faktor{\mathfrak{m} \mathbb C \{z,w\}^{N'} }{\mathfrak{m}^{k+1}}$
with $(z,w)\in \mathbb C^n_z \times \mathbb C^d_w=\mathbb C^N$. Although Lemma
\ref{secseg} shows that one can determine any solution of (\ref{lineq}) along
$\mathcal S^q_0$ by its $q\kappa$-th jet $\Lambda\in J_0^{q\kappa}$ at
$(0,0)$, it gives no information about which jets $\Lambda$ actually give rise
to a solution. Our aim now is to prove a jet parametrization result along the
same lines as Theorem \ref{jetparam}; however, we will also need to keep track
of the dependence of the solutions on the initial data $(H,Q)$. If a real-
analytic submanifold  $M$, defined in normal coordinates by the function $Q$
as above, is of finite commutator type, then for some  $\mathbf t<d+1$
the Segre map  $S^{\mathbf t}_0$ is generically finite by
the minimality criterion of  Baouendi, Ebenfelt and Rothschild. We can therefore
define a finite number  $\nu (M) = \nu (Q)$ as the minimum order of vanishing of
minor of maximal size of  the  Jacobian of $S^{\mathbf t}_0$.

\begin{prop}\label{linjetparamgen} Let $M\subset \mathbb C^N$ be a generic real-analytic
submanifold as above, $\mathcal M=\{w = \tilde Q(z,\chi,\tau)\}$, and $H$,$R$ as
above. Then there is a neighborhood $\Omega_2$ of $\tilde Q$ in $\mathbb
C\{z,\chi,\tau\}^d$, a neighborhood $\Omega_1$ of $H$ in $\mathbb
C\{z,w\}^{N'}$, and, for all $1\leq \ell,h \leq N'$,
 $0\leq |j_1|+|j_2|\leq \tau \kappa$, 
 continuous maps
   \[  \Omega_1\times \Omega_2\ni (H,Q) = p \to
K(p)^{h,j_1,j_2}_\ell = K^{h,j_1,j_2}_\ell(z,w) \in \mathbb C\{z,w\}, \]
\[\Omega_1\times \Omega_2\ni (H,Q) = p \to (a_{n_1}(p)^{h,j_1,j_2}_\ell,
b_{n_1}(p)^{h,j_1,j_2}_\ell)  \in \mathbb C^2, \ \ n_1\in \mathbb N,\] with
the following properties. For any given $(H,Q)\in \Omega_1\times \Omega_2$
with $Q$ satisfying $Q(z,\chi, \bar Q (\chi,z,w)) = w$ and
 $\nu (Q) = \nu (\tilde Q)$, there exists a solution
  $\alpha = (\alpha_1,\ldots,\alpha_{N'})\in \mathbb C\{z,w\}^{N'}$
   to (\ref{lineq}) whose
$\tau\kappa$-jet at $0$ is $\Lambda_0 = \Lambda_{0,h}^{j_1,j_2}\in
J_0^{\tau\kappa}$ if and only if
 \begin{equation}\label{eq:coeffexpand}
\sum_{\substack{0\leq |j_1| + |j_2|\leq \tau\kappa \\ 1\leq h \leq N'}}
(a_{n_1,\ell}^{h,j_1,j_2} \Lambda_{0,h}^{j_1,j_2} + b_{n_1,\ell}^{h,j_1,j_2}
\overline \Lambda_{0,h}^{j_1,j_2})=0 \ \ {\rm for\ all}\ n_1\in \mathbb N, \
1\leq \ell \leq N'. \end{equation}
 In this case, the (unique) solution
$\alpha$ is given by $\alpha=(K_1(z,w,\Lambda_0), \ldots,K_{N'}(z,w,\Lambda_0))$,
 where \[K_\ell(z,w,\Lambda_0):=\sum_{\substack{0\leq |j_1| + |j_2|\leq \tau\kappa \\ 1\leq h \leq N'}}  K^{h,j_1,j_2}_\ell(z,w)
\Lambda_{0,h}^{j_1,j_2} \] 
for all $1\leq \ell \leq N'$.
 \end{prop}

\begin{proof} As recalled above,  the Segre map $S^{\mathbf t}_0:\mathbb C^{\tau n}\to \mathbb C^N$ is generically of full rank, so that $\nu (\tilde Q) < \infty$.

We can thus appeal to Theorem~5 from \cite{JL}: as a special case of that result, we have that there exist a neighborhood $\mathcal V$ of $S^{\mathbf t}_0$ in $\mathbb C\{x^{[1;\mathbf t]}\}^N$ and a holomorphic map 
\[\Phi:\mathcal V\times \mathbb C\{x^{[1;\mathbf t]}\} \to \mathbb C\{z,w\}\]
such that $\Phi(A,g\circ A) = g$ for all $A\in \mathcal V$ with $\nu(A) = \nu(\tilde Q)$, and for all $g\in\mathbb C\{z,w\}$. Furthermore, the map $\Phi$ is linear in the second factor.

Since the map $\mathbb C\{z,\chi,\tau\}^d\ni Q \to S^{\mathbf t}_0(Q) \in \mathbb C\{x^{[1;\mathbf t]}\}^N$ is continuous, we can choose $\Omega_2$ in such a way that $S^{\mathbf t}_0(Q)\in \mathcal V$ for all $Q\in \Omega_2$. If we  select $\Omega_1$ small enough, we can also ensure that $s(0)\neq 0$ for all $(H,Q)\in \Omega_1\times \Omega_2$, where $s$ is given by  (\ref{folcon}). For $1\leq \ell \leq N'$ and for any $(H,Q)\in \Omega_1\times \Omega_2$, define then $\psi_\ell\in \mathbb C\{x^{[1;\mathbf t]}\}$ as the right hand side of (\ref{secondsegre}) for $q=\mathbf t$, that is,
\[\psi_{\ell}(\Lambda)(x^{[1;\mathbf t]})=\frac{1}{S} \sum_{\substack{0\leq |j_1| + |j_2| \leq \tau\kappa\\ 1\leq h \leq N'}} q_\ell^{h,j_1,j_2}(\partial^{\mathbf t\kappa} r^k, \partial^{\mathbf t\kappa} \overline r^k, \partial^{\mathbf t\kappa} \overline Q, \partial^{\mathbf t\kappa} Q)  \Lambda_h^{j_1,j_2}\]
for all $\Lambda\in J^{\mathbf t\kappa}_0$. We have that $\psi_\ell$ is analytic and depends continuously on $(H,Q)\in \Omega_1\times \Omega_2$ because $s(0)\neq 0$ (hence $S(0)\neq 0$ as well).

Finally, for all $(H,Q)\in \Omega_1\times \Omega_2$ set
 $K_\ell(z,w,\Lambda) = \Phi(S^{\mathbf t}_0(Q), \psi_{\ell}(\Lambda)(x^{[1;\mathbf t]}))$: 
 by the
properties of $\Phi$, we have that $K_\ell$ is linear in $\Lambda$, and is
continuous in $(H,Q)$. It follows from Lemma \ref{secseg} and from Theorem~5
in \cite{JL} that
 $\alpha_\ell(z,w) = K_\ell(z,w, j^{\mathbf t\kappa}_0\alpha)$
  whenever $\alpha$ is a solution of (\ref{lineq}). The
remaining statement can be proved by setting
 $\alpha_\ell =
K_\ell(z,w,\Lambda)$ in (\ref{lineq}) and expanding it as a power series in
$(z,\chi,\tau)$: the coefficients of this power series depend continuously on
$(H,Q)\in \Omega_1\times \Omega_2$ and linearly on
 $\Lambda,\overline
\Lambda$, so that the linear equations (\ref{eq:coeffexpand}) can be obtained
by setting all the coefficients to $0$ (see also the proof of Prop.
\ref{linjetparam}).  \end{proof}

The result above allows to deduce the following generalization of Theorem~2.1 in \cite{CH}:

\begin{cor}
\label{cor:chohangeneral}
Let $M$ be a generic real-analytic submanifold of $\mathbb C^N$, $0\in M$, and let $H$ be a $\kappa$-nondegenerate holomorphic embedding of $M$ into a generic real-analytic submanifold $M'\subset \mathbb C^{N'}$, $N'\geq N$. Suppose that the Segre map $S^{\mathbf t}_0$ is 
generically finite. Then the space of infinitesimal deformations of $H$ is finite dimensional, and any infinitesimal deformation is determined by its $\mathbf t\kappa$-jet at $0$. In particular, 
any infinitesimal deformations of $H$ is determined by its  $(d+1)\kappa$-jet at $0$, and
if $M$ is strictly pseudoconvex, any infinitesimal deformations 
of $H$ is determined by its  $2 \kappa$-jet at $0$ .
\end{cor}

The last two statements follow immediately because $\mathbf t \leq d+1$ and $\mathbf t = 2$ if 
$M$ is strictly pseudoconvex. 

We are now interested in giving another version of the proof of the Proposition \ref{linjetparamgen}. The reason is that, although the second proof is less general than the previous one, it is more suited to concrete computations and gives the outline of the algorithmic steps used later for specific examples. In order to achieve this we will closely follow the proofs of  \cite[Propositions 2.11, 3.1]{BER}, and turn back to the case where $N=2$, $N'=3$, $\kappa=2$ and $M$ is strongly pseudoconvex.

We will denote by $\Omega_1$ the open subset of $(\mathbb C\{z,w\})^3$ consisting of those $H$  which satisfy (\ref{folcon}), and by $\Omega_2$ the open subset of $\mathbb C\{z,\chi,\tau\}$ consisting of those $Q$ which satisfy $Q_{z\chi}(0,0,0)\neq 0$ (i.e.\ such that $M = \{w = Q(z,\chi,\tau)\} $ is strongly pseudoconvex around $0$).

\begin{prop}\label{linjetparam}
In what follows, the indices $j_1,j_2,h,\ell$ satisfy $0\leq j_1 + j_2 \leq 4$ and $1\leq h,\ell \leq 3$. Fix $R \subset (\mathbb C\{z,w,\chi,\tau\})^3$ as before. 
There exist
\begin{itemize}
\item continuous maps 
\[\Omega_1\times \Omega_2\ni (H,Q) = p \to \Psi(p)^{h,j_1,j_2}_\ell = \Psi^{h,j_1,j_2}_\ell(z,t) \in \mathbb C\{z,t\},\]
\[\Omega_1\times \Omega_2\ni (H,Q) = p \to K(p)^{h,j_1,j_2}_\ell = K^{h,j_1,j_2}_\ell(z,w) \in \mathbb C\{z,w\};\]
\item continuous maps
\[\Omega_2 \ni Q \to B_Q = B(z) \in \mathbb C\{z\},\]
\[\Omega_1\times \Omega_2\ni (H,Q) = p \to (d_n(p)^{h,j_1,j_2}_\ell, e_n(p)^{h,j_1,j_2}_\ell, f_n(p)^{h,j_1,j_2}_\ell, g_n(p)^{h,j_1,j_2}_\ell) \in \mathbb C^4, \ \ �n\in \mathbb N\]
\end{itemize}
such that, for any given $(H,Q)\in \Omega_1\times \Omega_2$, there exists a solution $\alpha = (\alpha_1,\alpha_2,\alpha_3)\in (\mathbb C\{z,w\})^3$ to (\ref{lineq}) whose $4$-th jet at $0$ is $\Lambda_0 = \Lambda_{0,h}^{j_1,j_2}\in J_0^4$ if and only if the following conditions hold:
\begin{itemize}
\item[(i)] the maps 
\begin{equation}\label{formalpsi}
(z,w) \to \sum_{\substack{0\leq j_1 + j_2\leq 4\\ 1\leq h \leq 3}}  \Psi^{h,j_1,j_2}_\ell \left (z,\frac{w}{B(z)} \right ) \Lambda_{0,h}^{j_1,j_2}, \ 1\leq \ell \leq 3, 
\end{equation} 
extend holomorphically to a neighborhood of $0$ in $\mathbb C^2$. This happens if and only if 
\[\sum_{\substack{0\leq j_1 + j_2\leq 4\\1\leq h \leq 3 }} d_{n,\ell}^{h,j_1,j_2} \Lambda_{0,h}^{j_1,j_2}=0 \ \ {\rm for\ all}\ n\in \mathbb N, \ 1\leq \ell \leq 3\]
and if and only if 
\[\sum_{\substack{0\leq j_1 + j_2\leq 4\\ 1\leq h \leq 3}}  \Psi^{h,j_1,j_2}_\ell \left (z,\frac{w}{B(z)} \right ) \Lambda_{0,h}^{j_1,j_2} = \sum_{\substack{0\leq j_1 + j_2\leq 4\\ 1\leq h \leq 3}}  K^{h,j_1,j_2}_\ell(z,w)  \Lambda_{0,h}^{j_1,j_2} =: K_\ell(z,w,\Lambda_0);\]
\item[(ii)] the $4$-jet of the map  $(z,w) \to (K_1(z,w,\Lambda_0), K_2(z,w,\Lambda_0), K_3(z,w,\Lambda_0) )$ at $0$ coincides with $\Lambda_0$. If (i) is satisfied, this holds if and only if 
\[\sum_{\substack{0\leq j_1 + j_2\leq 4\\ 1\leq h \leq 3 }} e_{n,\ell}^{h,j_1,j_2} \Lambda_{0,h}^{j_1,j_2}=0 \ \ {\rm for\ all}\ n\in \mathbb N, \ 1\leq \ell \leq 3;\]
\item[(iii)] The triple $(K_1(z,w,\Lambda_0), K_2(z,w,\Lambda_0), K_3(z,w,\Lambda_0))$ solves (\ref{lineq}).    If (i) is satisfied, this holds if and only if
\[\sum_{\substack{0\leq j_1 + j_2\leq 4 \\ 1\leq h \leq 3}} (f_{n,\ell}^{h,j_1,j_2} \Lambda_{0,h}^{j_1,j_2} + g_{n,\ell}^{h,j_1,j_2} \overline \Lambda_{0,h}^{j_1,j_2})=0 \ \ {\rm for\ all}\ n\in \mathbb N, \ 1\leq \ell \leq 3.\]
\end{itemize} 
\end{prop}
\begin{proof} As mentioned before, we are going to follow the proof of some propositions in \cite{BER}: we start with Proposition 2.11. Write $Q(z,\chi,0) = \sum_{j\geq 1} A_j(z)\chi^j$, define $C_j(z) = A_j(z) A_1(z)^{j-2}$ for $j\geq 2$, and let $\hat \psi(z,u) := u + \sum_{j\geq 2} C_j(z)u^j$. It is clear that each $A_j$ and $\hat \psi$ depend continuously on $Q$. It follows that the series $\psi(z,t) = t + \sum v_j(z) t^j$, obtained by solving for $u$ the equation $t = \hat\psi(z,u)$ by the implicit function theorem, also depends continuously on $Q$ (see Remark \ref{implfunct}). By construction, $w \equiv Q(z, A_1(z)\psi(z,w/A_1(z)^2),0 )$.

Put now
\[\phi_\ell^{h,j_1,j_2}(z,\chi) = \frac{1}{S}  q_\ell^{h,j_1,j_2}(\partial^4 r^k, \partial^4 \overline r^k, \partial^4 \overline Q,  \partial^4 Q)  \]
where $S=s{\overline s}^3(0), q_\ell^{h,j_1,j_2}$ are given by Lemma \ref{secseg}. Because of the polynomial form of  $s,\overline s, q_\ell^{h,j_1,j_2}$ and the analyticity of $r^k$, we have that the map $\Omega_1\times \Omega_2 \ni (H,Q) \to \phi_\ell^{h,j_1,j_2} \in \mathbb C\{z,\chi\}$ is continuous. It follows that defining
\[\Psi^{h,j_1,j_2}_\ell(z,t) = \phi_\ell^{h,j_1,j_2}(z, A_1(z)\psi(z,t) )\]
and $B(z) = A_1(z)^2$, the first statement in (i) is verified. 

To verify the other points, we now look at the proof of Proposition 3.1 in \cite{BER}. Write
\[B(z)^j = z^{2j}U^j(z) \ {\rm with}\ U(0)\neq 0,\]
\[\Psi_\ell^{h,j_1,j_2}(z,t) = \sum_{m_1,m_2\geq 0}\psi_{\ell,m_1,m_2}^{h,j_1,j_2}z^{m_1}t^{m_2}, \]
\[\Psi_\ell^{h,j_1,j_2}\left (z,\frac{t}{B(z)} \right ) = \sum_{m_1,m_2\geq 0}\psi_{\ell,m_1,m_2}^{h,j_1,j_2}\frac{z^{m_1-2m_2}}{U^{m_2}(z)}t^{m_2} = \sum_{m_1,m_2\geq 0}\Psi_{\ell,m_1,m_2}^{h,j_1,j_2}z^{m_1 - 2m_2}t^{m_2}; \]
once again, $U$, $\psi_{\ell,m_1,m_2}^{h,j_1,j_2}$ and $\Psi_{\ell,m_1,m_2}^{h,j_1,j_2}$ depend on $(H,Q)\in \Omega_1\times \Omega_2$ in a (uniformly) continuous way. Note that we have used the fact that $A_1(z) = Q_\chi(z,0,0) = az + O(z^2)$ for some $a\neq 0$, which holds because $Q\in \Omega_2$.

With these definitions, the second and third statements of (i) follow by putting
\[K^{h,j_1,j_2}_\ell(z,w) =  \sum_{m_1 \geq 2m_2}\Psi_{\ell,m_1,m_2}^{h,j_1,j_2}z^{m_1 - 2m_2}w^{m_2} ,  \]
\[d^{h,j_1,j_2}_{n,\ell} =  \Psi_{\ell,\iota_1(n),\iota_2(n)}^{h,j_1,j_2} \]
where $\mathbb N \ni n \to (\iota_1(n),\iota_2(n))\in  \{(m_1,m_2)\in \mathbb N^2: m_1 < 2m_2\}$ is any bijection, since
\[\sum_{\substack{0\leq j_1 + j_2\leq 4\\ 1\leq h \leq 3}} \Psi_\ell^{h,j_1,j_2}\left (z,\frac{w}{B(z)} \right ) \Lambda_{0,h}^{j_1,j_2} = \]
\[ =  \sum_{\substack{0\leq j_1 + j_2\leq 4\\ 1\leq h \leq 3\\ m_1\geq 2m_2}} \Psi_{\ell,m_1,m_2}^{h,j_1,j_2} z^{m_1 - 2m_2}w^{m_2} \Lambda_{0,h}^{j_1,j_2}  +  \sum_{\substack{0\leq j_1 + j_2\leq 4\\ 1\leq h \leq 3\\ m_1 < 2m_2}} \Psi_{\ell,m_1,m_2}^{h,j_1,j_2} z^{m_1 - 2m_2}w^{m_2} \Lambda_{0,h}^{j_1,j_2} = \]
\[= \sum_{\substack{0\leq j_1 + j_2\leq 4\\ 1\leq h \leq 3}}  K^{h,j_1,j_2}_\ell(z,w)  \Lambda_{0,h}^{j_1,j_2} + \sum_{m_2\geq 0}w^{m_2} \sum_{m_1<2m_2}\frac{1}{z^{2m_2-m_1}}\sum_{\substack{0\leq j_1 + j_2\leq 4\\ 1\leq h \leq 3}}  \Psi_{\ell,m_1,m_2}^{h,j_1,j_2}  \Lambda_{0,h}^{j_1,j_2}   \]
and the last line extends holomorphically around $(0,0)$ if and only if the second summand vanishes identically.

The points (ii) and (iii) follow in a very similar way by, respectively, comparing the $4$-th jet of the map  $(z,w) \to (K_1(z,w,\Lambda_0), K_2(z,w,\Lambda_0), K_3(z,w,\Lambda_0) )$ with $\Lambda_0$ and by inserting the triple $(K_1,K_2,K_2)$ in the left-hand side of (\ref{lineq}) and then expanding it as a power series in $(z,\chi,\tau)$ (cf.\ again \cite[Proposition 3.1]{BER}).
\end{proof}
Despite the rather technical nature of the statement of Proposition \ref{linjetparam}, we can deduce the following interesting consequence:
\begin{cor}\label{semicont}
Fix $R\subset (\mathbb C\{z,w,\chi,\tau\})^3$. For any $(H,Q)\in \Omega_1\times \Omega_2$, the dimension $d(H,Q)$ of the space of the solutions of (\ref{lineq}) is finite. Moreover, the function $d:\Omega_1\times \Omega_2\to \mathbb N$ is upper semicontinuous, i.e.\ for any $p=(H,Q)\in \Omega_1\times \Omega_2$ there exists a neighborhood $\mathcal V$ of $p$ in $\Omega_1\times \Omega_2$ such that for any $p'=(H',Q')\in \mathcal V$ we have $d(p')\leq d(p)$.
\end{cor}
\begin{proof}
For $p=(H,Q)\in \Omega_1\times \Omega_2$, let $\mathcal L(p)$ be the space of the solutions of (\ref{lineq}), and let $L(p)\subset J_0^4$ be the (real) linear subspace defined by
\[\sum_{\substack{0\leq j_1 + j_2\leq 4\\1\leq h \leq 3 }} d_{n,\ell}^{h,j_1,j_2} \Lambda_{h}^{j_1,j_2}= \sum_{\substack{0\leq j_1 + j_2\leq 4\\1\leq h \leq 3 }} e_{n,\ell}^{h,j_1,j_2} \Lambda_{h}^{j_1,j_2}= \] 
\[ = \sum_{\substack{0\leq j_1 + j_2\leq 4 \\ 1\leq h \leq 3}} (f_{n,\ell}^{h,j_1,j_2} \Lambda_{h}^{j_1,j_2} + g_{n,\ell}^{h,j_1,j_2} \overline \Lambda_{h}^{j_1,j_2})=0 \ \ {\rm for\ all}\ n\in \mathbb N, \ 1\leq \ell \leq 3; \]
by Theorem \ref{linjetparam}, the projection $j_0^4: (\mathbb C\{z,w\})^3 \to J_0^4$ induces an isomorphism between $\mathcal L(p)$ and $L(p)$, thus $\dim_{\mathbb R} \mathcal L(p)<\infty$. For the second statement, identify $(d_{n,\ell}^{h,j_1,j_2} , 0)$, $(0,\overline d_{n,\ell}^{h,j_1,j_2}) $, $(e_{n,\ell}^{h,j_1,j_2} , 0)$, $(0,\overline e_{n,\ell}^{h,j_1,j_2}) $, $(f_{n,\ell}^{h,j_1,j_2}, g_{n,\ell}^{h,j_1,j_2} ) $ with vectors in $J_0^4 \times J_0^4$; the space they span is orthogonal to (the complexification of) $L(p)$ and can be generated by finitely many independent ones, say $v_1(p),\ldots, v_k(p)$. Since each $v_j(p)$ depends continuously on $p$, choosing a small neighborhood $\mathcal V$ of $p$ the vectors $v_1(p'),\ldots, v_k(p')$ are independent also for $p'\in \mathcal V$, hence $\dim_{\mathbb R}L(p')\leq \dim_{\mathbb R}L(p)$.
\end{proof}

\section{Examples}\label{sec:examples}

\begin{example}\label{modelexample}
We consider $M_0=\{{\rm Im}w - |z|^2 -\varepsilon |z|^4 = 0\}$ and show that $H_0: (z,w)\mapsto (z,z^2,w)$ is a locally rigid embedding from $M_0$ to $\mathbb H^3_\varepsilon$. Here we restrict ourselves to study the sphere case where $\varepsilon = +1$, since the case of $\mathbb H^3_-$ as target works in the same way. In the case of positive signature the equation \eqref{lineq} becomes
\begin{equation}\label{lineqz2}
 \alpha_{0,3}(z,w) - \overline \alpha_{0,3}(\chi,\tau) - 2 i (\alpha_{0,1}(z,w)\chi + \overline \alpha_{0,1}(\chi,\tau) z  +  \alpha_{0,2}(z,w) \chi^2 + \overline \alpha_{0,2}(\chi,\tau) z^2) = 0,
\end{equation}
if $w=\tau + 2 i (z \chi + z^2 \chi^2)$, where we solve for $\alpha_0(z,w)=(\alpha_{0,1}(z,w),\alpha_{0,2}(z,w),\alpha_{0,3}(z,w))\in  (\mathbb C\{z,w\})^3$ with $\alpha_{0,j}(0,0)=0$. For the jet parametrization as described in the proof of Proposition \ref{linjetparam} we have with the notation borrowed from there, $A_1(z) = 2 i z$ and $B(z)= A_1(z)^2$, such that 
\begin{align*}
A_1(z) \psi(z,w/B(z)) = - (1- \sqrt{1 - 2 i w})/(2 z),
\end{align*} 
where we choose the principal branch of the square root. For $\Psi_{0,\ell}$ from \eqref{formalpsi} we set $\Psi_0(z,w) = (\Psi_{0,1}(z,w),\Psi_{0,2}(z,w),\Psi_{0,3}(z,w)) \in  (\mathbb C[[z,w]])^3$ and write $\Psi_{0,\ell}^{m,n}$ for the coefficient of $z^m w^n$ of $\Psi_{0,\ell}$. We denote $\Lambda_{0,j}^{m,n} = \frac{\partial^{m+n}\alpha_{0,j}}{\partial z^m \partial w^n}(0,0)$ and we scale some $\Lambda_{0,j}^{m,n}$ in order to avoid large numerical factors. The equations from (i) of Proposition \ref{linjetparam} are given by the vanishing of the following coefficients of $\Psi_0$, which we denote by $S_1(\Psi_0)$:
\begin{align*}
\Psi_{0,1}^{-1,4} & =   3 \Lambda_{0,2}^{0,2} + 2 i \Lambda_{0,2}^{0,3},\\
\Psi_{0,3}^{-1,4} & = 12 \Lambda_{0,1}^{0,2} - 2 \Lambda_{0,1}^{0,3} + 3 \Lambda_{0,2}^{1,1} + 3 i \Lambda_{0,2}^{1,2},\\
\Psi_{0,1}^{-1,5} & =  12 \Lambda_{0,2}^{0,2} + 7 i \Lambda_{0,2}^{0,3} - \Lambda_{0,2}^{0,4},\\
\Psi_{0,3}^{-1,5} & = 18 \Lambda_{0,1}^{0,2} + 4 i \Lambda_{0,1}^{0,3} - \Lambda_{0,1}^{0,4} - 6 i \Lambda_{0,2}^{1,1} + 3 \Lambda_{0,2}^{1,2} + i \Lambda_{0,2}^{1,3},\\
\Psi_{0,1}^{-1,6} & = 75 \Lambda_{0,2}^{0,2} + 24 i \Lambda_{0,2}^{0,3} - 2 \Lambda_{0,2}^{0,4},\\
\Psi_{0,3}^{-1,6} & = 54 \Lambda_{0,1}^{0,2} + 6 i \Lambda_{0,1}^{0,3} - \Lambda_1^{0,4} - 21 i \Lambda_{0,2}^{1,1} + 4 \Lambda_{0,2}^{1,2} + i \Lambda_{0,2}^{1,3},\\
\Psi_{0,3}^{-1,7} & = 42 \Lambda_{0,1}^{0,2} + i \Lambda_{0,1}^{0,3} - 18 i \Lambda_{0,2}^{1,1}.
\end{align*}
The equations given in (ii) of Proposition \ref{linjetparam} are equivalent to $\Psi_0$ satisfying
the following equations, which  we abbreviate by $S_2(\Psi_0) =0$:
\begin{align*}
\Lambda_{0,2}^{2,1} -2 \Psi_{0,2}^{2,1} & =  6 \Lambda_{0,1}^{1,1} + 6 i \Lambda_{0,1}^{1,2} + 3 \Lambda_{0,2}^{2,1} - 6 \Lambda_{0,3}^{0,2} - i \Lambda_{0,3}^{0,3},\\
 \Lambda_{0,1}^{2,2} - 4 \Psi_{0,1}^{2,2} & = 3 i \Lambda_{0,1}^{2,1} + 3 \Lambda_{0,1}^{2,2} - 2 \Lambda_{0,3}^{1,3},\\
\Lambda_{0,2}^{2,2} - 4 \Psi_{0,2}^{2,2} & = 6 \Lambda_{0,1}^{1,1} - 6 i \Lambda_{0,1}^{1,2} - 2 \Lambda_{0,1}^{1,3} - 3 \Lambda_{0,2}^{2,1} - i \Lambda_{0,3}^{0,3} + \Lambda_{0,3}^{0,4},\\
\Lambda_{0,1}^{1,3} - 6 \Psi_{0,1}^{1,3}  & = 18 \Lambda_{0,1}^{1,1} - 8 \Lambda_{0,1}^{1,3} - 3 \Lambda_{0,2}^{2,1} + 3 i \Lambda_{0,2}^{2,2} - 6 \Lambda_{0,3}^{0,2} - 6 i \Lambda_{0,3}^{0,3} + 3 \Lambda_{0,3}^{0,4},\\
 \Lambda_{0,2}^{1,3} - 6 \Psi_{0,2}^{1,3} & = 42 i \Lambda_{0,1}^{0,2} - 6 \Lambda_{0,1}^{0,3} -3 i \Lambda_{0,1}^{0,4} + 18 \Lambda_{0,2}^{1,1} - \Lambda_{0,2}^{1,3}.
\end{align*}
We add the following equations to $S_2(\Psi_0)$. In fact they are not needed for this case, since the elements $\Lambda_{0,m}^{k,\ell}$ of $J_0^4$ appearing on the left-hand side do not occur in $\Psi_0$, but we use these equations in section \ref{sec:genericity} below. In addition to the 
equations above, $S_2(\Psi_0) = 0$ now also means that 
$\Psi_0 $ satisfies the following equations:
\begin{align*}
\Lambda_{0,3}^{1,0}  - \Psi_{0,3}^{1,0}& = \Lambda_{0,3}^{1,0}, \\
\Lambda_{0,3}^{2,0}  - 2 \Psi_{0,3}^{2,0}& = \Lambda_{0,3}^{2,0}, \\
\Lambda_{0,1}^{3,0}  - 6 \Psi_{0,1}^{3,0}& =  \Lambda_{0,1}^{3,0} - 3 \Lambda_{0,3}^{2,1}, \\
\Lambda_{0,2}^{3,0}  - 6 \Psi_{0,2}^{3,0}& =  6 i \Lambda_{0,1}^{2,1} + \Lambda_{0,2}^{3,0} - 6 \Lambda_{0,3}^{1,1} - 6 i \Lambda_{0,3}^{1,2}, \\
\Lambda_{0,3}^{3,0}  - 6 \Psi_{0,3}^{3,0}& = \Lambda_{0,3}^{3,0}, \\
\Lambda_{0,1}^{4,0}  - 24 \Psi_{0,1}^{4,0}& = \Lambda_{0,1}^{4,0}, \\
\Lambda_{0,2}^{4,0}  - 24 \Psi_{0,2}^{4,0}& = \Lambda_{0,2}^{4,0} - 12 \Lambda_{0,3}^{2,1} + 12 i \Lambda_{0,3}^{2,2}, \\
\Lambda_{0,3}^{4,0}  - 24 \Psi_{0,3}^{4,0}& = \Lambda_{0,3}^{4,0}, \\
\Lambda_{0,1}^{3,1}  - 6 \Psi_{0,1}^{3,1}& = \Lambda_{0,1}^{3,1} - 3 \Lambda_{0,3}^{2,2}, \\
\Lambda_{0,2}^{3,1}  - 6 \Psi_{0,2}^{3,1}& = 9 \Lambda_{0,1}^{2,1} - 3 i \Lambda_{0,1}^{2,2} + \Lambda_{0,2}^{3,1} - 6 \Lambda_{0,3}^{1,2} + 3 i\Lambda_{0,3}^{1,3}, \\
\Lambda_{0,3}^{3,1}  - 6 \Psi_{0,3}^{3,1}& = \Lambda_{0,3}^{3,1}.
\end{align*}
We write $\rho_0(z,\chi,\tau)$ for the left-hand side of \eqref{lineqz2} and set $\rho_{0,k,\ell,m} = \frac{\partial^{k+\ell+m} \rho_0}{\partial z^k \partial \chi^{\ell} \partial \tau^m}(0)$. Then the equations from (iii) of Proposition \ref{linjetparam} are given by the vanishing of the left hand side of the following equations:
\begin{align*}
\rho_{0,0,1,1} & =2 \Lambda_{0,1}^{0,1} - i \overline{\Lambda}_{0,3}^{1,1},\\
\rho_{0,1,1,0} & = \Lambda_{0,3}^{0,1} - \Lambda_{0,1}^{1,0} -  \overline{\Lambda}_{0,1}^{1,0},\\
\rho_{0,1,2,0} & = 4 i \Lambda_{0,1}^{0,1} + 2 \Lambda_{0,2}^{1,0} + \overline{\Lambda}_{0,1}^{2,0},\\
\rho_{0,1,2,1} & = 4 i \Lambda_{0,1}^{0,2} + \Lambda_{0,2}^{1,1} + \overline{\Lambda}_{0,1}^{2,1},\\
\rho_{0,1,3,0} & = 4 i \Lambda_{0,2}^{0,1} + \overline{\Lambda}_{0,3}^{2,1},\\
\rho_{0,2,2,0} & = 4 i \Lambda_{0,1}^{1,1} + \Lambda_{0,2}^{2,0} + \overline{\Lambda}_{0,2}^{2,0} - 2 \Lambda_{0,3}^{0,1} - 2 i \Lambda_{0,3}^{0,2},\\
\rho_{0,2,2,1} & = 4 i \Lambda_{0,1}^{1,2} + \Lambda_{0,2}^{2,1} - 6 \overline{\Lambda}_{0,1}^{1,1} + 6 i \overline{\Lambda}_{0,1}^{1,2} - 2 \overline{\Lambda}_{0,2}^{2,1} + 6 \overline{\Lambda}_{0,3}^{0,2} - i \overline{\Lambda}_{0,3}^{0,3} - 2 \Lambda_{0,3}^{0,2} - i \Lambda_{0,3}^{0,3},\\
\rho_{0,2,3,0} & = 2 \Lambda_{0,1}^{0,1} + 2 i \Lambda_{0,1}^{0,2} + \Lambda_{0,2}^{1,1} + \overline{\Lambda}_{0,1}^{2,1} - i \overline{\Lambda}_{0,3}^{1,1} - \overline{\Lambda}_{0,3}^{1,2},\\
\rho_{0,2,3,1} & = 52 \Lambda_{0,1}^{0,2} + 10 i \Lambda_{0,1}^{0,3} - 12 i \Lambda_{0,2}^{1,1} + 14 \Lambda_{0,2}^{1,2} + 3 i \overline{\Lambda}_{0,1}^{2,1} - \overline{\Lambda}_{0,1}^{2,2} - 2 i \overline{\Lambda}_{0,3}^{1,2} + \overline{\Lambda}_{0,3}^{1,3},\\
\rho_{0,2,4,0} & = 4 \Lambda_{0,2}^{0,1} + i \Lambda_{0,2}^{0,2} - i \overline{\Lambda}_{0,3}^{2,1} + \overline{\Lambda}_{0,3}^{2,2},\\
\rho_{0,3,4,3} & = 1210 \Lambda_{0,1}^{0,2} - 108 i \Lambda_{0,1}^{1,2} + 27 \Lambda_{0,1}^{0,4} - 580 i \Lambda_{0,2}^{1,1} - 119 \Lambda_{0,2}^{1,2} - 27 i \Lambda_{0,2}^{1,3}.
\end{align*}
The collection of the above derivatives of $\rho_0$ we denote by $S_{3.1}(\rho_0)$ and the following derivatives of $\rho_0$ we denote by $S_{3.2}(\rho_0)$ and define $S_3(\rho_0)=\{S_{3.1}(\rho_0),S_{3.2}(\rho_0)\}$:
\begin{align*}
\rho_{0,1,1,1} & = \Lambda_{0,1}^{1,1} + \overline{\Lambda}_{0,1}^{1,1} - \Lambda_{0,3}^{0,2},\\
\rho_{0,2,2,2} & = 2 i \Lambda_{0,1}^{1,3}  - 2 i \overline{\Lambda}_{0,1}^{1,3} + \Lambda_{0,2}^{2,2} + \overline{\Lambda}_{0,2}^{2,2} - \overline{\Lambda}_{0,3}^{0,3}  - \Lambda_{0,3}^{0,3} + i \overline{\Lambda}_{0,3}^{0,4} - i \Lambda_{0,3}^{0,4} + 6 i \overline{\Lambda}_{0,1}^{1,1} - 6 \overline{\Lambda}_{0,1}^{1,2} - 3 i \overline{\Lambda}_{0,2}^{2,1} ,\\
\rho_{0,3,3,3} & = 10 \Lambda_{0,1}^{1,1} - 2 i \Lambda_{0,1}^{1,2} - \Lambda_{0,2}^{2,1} - 4 \Lambda_{0,3}^{0,2},\\
\rho_{0,3,3,4} & = 42 i \Lambda_{0,1}^{1,1} + 26 \Lambda_{0,1}^{1,2} + 2 i \Lambda_{0,1}^{1,3} - 9 i \Lambda_{0,2}^{2,1} + \Lambda_{0,2}^{2,2} - 12 i \Lambda_{0,3}^{0,2} - 5 \Lambda_{0,3}^{0,3} - i \Lambda_{0,3}^{0,4},\\
\rho_{0,4,4,3} & = 60 i \Lambda_{0,1}^{1,1} + 28 \Lambda_{0,1}^{1,2} + 2 i \Lambda_{0,1}^{1,3} - 10 i \Lambda_{0,2}^{2,1} +  \Lambda_{0,2}^{2,2} - 20 i \Lambda_{0,3}^{0,2} - 5 \Lambda_{0,3}^{0,3} - i \Lambda_{0,3}^{0,4},\\
\rho_{0,5,5,3} & = 42 i \Lambda_{0,1}^{1,1} + 28 \Lambda_{0,1}^{1,2} + 2 i \Lambda_{0,1}^{1,3} - 9 i \Lambda_{0,2}^{2,1} + \Lambda_{0,2}^{2,2} - 12 i \Lambda_{0,3}^{0,2} - 6 \Lambda_{0,3}^{0,3} - i \Lambda_{0,3}^{0,4}.
\end{align*}
Let  us denote by $\lambda_0 \in \mathbb C^{74}$ the vector consisting of the following elements of the $4$-jet of $\alpha_0$ at $0$, given by
\begin{align*}
\Lambda_{0,3}^{1,0}, \Lambda_{0,3}^{0,1}, \Lambda_{0,3}^{2,0}, \Lambda_{0,1}^{2,0}, \Lambda_{0,2}^{1,1}, \Lambda_{0,3}^{1,1}, \Lambda_0^{0,2}, \Lambda_0^{3,0}, \Lambda_0^{2,1}, \Lambda_0^{1,2}, \Lambda_0^{0,3}, \Lambda_0^{4,0}, \Lambda_0^{2,2},\Lambda_0^{3,1}, \Lambda_0^{1,3}, \Lambda_0^{0,4},
\end{align*}
and its conjugates and the derivatives $\overline{\Lambda}_{0,1}^{1,1}$ and $\overline{\Lambda}_{0,2}^{2,0}$. Above we use the notation $\Lambda_0^{m,n}=(\Lambda_{0,1}^{m,n},\Lambda_{0,2}^{m,n},\Lambda_{0,3}^{m,n})$. Then we consider the following collection of linear expressions for $\lambda_0$, which consists of $74$ components
\begin{align*}
S(\Psi_0,\rho_0)=\{S_1(\Psi_0),\overline{S}_1(\Psi_0),S_2(\Psi_0),\overline{S}_2(\Psi_0),S_{3.1}(\rho_0),\overline{S}_{3.1}(\rho_0), S_{3.2}(\rho_0)\},
\end{align*}
where $\overline{S}_k(\Phi)$ means that we conjugate all equations in $S_k(\Phi)$. Then we compute that the Jacobian matrix of $S$ with respect to $\lambda_0$ is of full rank. The remaining derivatives of $\alpha_0$ at $0$ which belong to $J_0^4$ and occur in $S(\Psi_0,\rho_0)$ are the following variables
\begin{align*}
\mu_0=\left(\Lambda_{0,1}^{1,0}, \overline{\Lambda}_{0,1}^{1,0}, \Lambda_{0,1}^{0,1}, \overline{\Lambda}_{0,1}^{0,1}, \Lambda_{0,2}^{1,0},\overline{\Lambda}_{0,2}^{1,0}, \Lambda_{0,2}^{0,1}, \overline{\Lambda}_{0,2}^{0,1}, \Lambda_{0,1}^{1,1}, \Lambda_{0,2}^{2,0}\right)\in \mathbb C^{10}.
\end{align*}
Thus we obtain $\dim_{\mathbb R} \mathfrak {hol}_0 (\Phi(\Lambda_0)) = 10$ and by Theorem \ref{suffcon2} we obtain local rigidity of $H_0$.
\end{example}

\begin{example}\label{itworks}
We define
\begin{align*}
M_1 &=\{(z,w)\in \mathbb C^2: {\rm Im} w  = |z|^2 + 3 |z|^4 + 2 {\rm Re}(z^3 \overline z^4) + 2 {\rm Re}(z^6 \overline z)\},\\
M_1' & =\{(z_1',z_2',w') \in \mathbb C^3: {\rm Im} w' = |z_1'|^2 + |z_2'|^2  + 2 {\rm Re}({z'_1}^2 \overline z'_2) + 2 {\rm Re}(z'_1 {{\overline z}'_2}^3) + 2 {\rm Re}({z_1'}^3 {{\overline z}'_2}^2)\},
\end{align*}
such that $H_1: (z,w) \mapsto (z,z^2,w)$ is a transversal and $2$-nondegenerate embedding from $M_1$ into $M'_1$. One can show that both $M_1$ and $M'_1$ do not possess any infinitesimal automorphisms which fix $0$. A similar computation as in Example \ref{modelexample} shows that the equation \eqref{lineq} in this case given by
\begin{align*}
 {\rm Re}(i \gamma_1(z,w) + (\overline z + 2 z \overline z^2 + 3 z^2 \overline z^4 + \overline z^6) \alpha_1(z,w) + (2 \overline z^2 + 2 z^2 \overline z^3 + 3 z^4 \overline z) \beta_1(z,w)) = 0,
\end{align*}
for $(z,w)\in M_1$, does not admit any nontrivial solution $(\alpha_1(z,w),\beta_1(z,w),\gamma_1(z,w)) \in (\mathbb C\{z,w\})^3$ fixing $0$.  According to Theorem \ref{suffcon1} the embedding $H_1: M_1 \rightarrow M_1'$ is locally rigid.
\end{example}

\begin{example}\label{itdoesntwork}
For the hypersurfaces $M_2$ and $M'_2$ given by
\begin{align*}
M_2 &=\{(z,w)\in \mathbb C^2: {\rm Im} w  = |z|^2 + 3 |z|^4 + 2 {\rm Re}(z^6 \overline z)\},\\
M_2' & =\{(z_1',z_2',w') \in \mathbb C^3: {\rm Im} w' = |z_1'|^2 + |z_2'|^2  + 2 {\rm Re}({z'_1}^2 \overline z'_2) + 2 {\rm Re}(z'_1 {{\overline z}'_2}^3) \},
\end{align*}
the spaces of infinitesimal automorphisms fixing $0$ are trivial. Moreover the map $H_2: (z,w) \mapsto (z,z^2,w)$ embeds $M_2$ into $M'_2$. The linear equation \eqref{lineq}, which in this case is given by
\begin{align*}
 {\rm Re}(i \gamma_2(z,w) + (\overline z + 2 z \overline z^2 + \overline z^6) \alpha_2(z,w) + (2 \overline z^2 + 3 z^4 \overline z) \beta_2(z,w)) = 0,
\end{align*}
for $(z,w) \in M_2$, has a nontrival solution given by $(\alpha_2(z,w),\beta_2(z,w),\gamma_2(z,w))=(i z, i z^2/3,0)$. One can show that the space of solutions is $1$-dimensional by following the procedure given in Proposition \ref{linjetparam}. So far we do not know whether the map $H_2: M_2 \rightarrow M_2'$ is locally rigid.
\end{example}

\begin{example}\label{nonsphericalfamily}
We define the following hypersurfaces
\begin{align*}
M_3 &=\{(z,w)\in \mathbb C^2: {\rm Im} w  = |z|^2 + |z|^4 \},\\
M_3' & =\{(z_1',z_2',w') \in \mathbb C^3: {\rm Im} w' = |z_1'|^2 + |z_2'|^2  + |z_1'|^4 + |z_1'|^2 |z_2'|^2 + {\rm Im} ({z_2'}^2 \bar z_1') + {\rm Im} (z_2' \bar z_1' \bar z_2')\}.
\end{align*}
The hypersurface $M_3$, which is $M_0$ from Example \ref{modelexample}, admits a real $1$-dimensional isotropy group fixing $0$ given by $(z,w) \mapsto (u z, w)$, where $|u|=1$, while for $M_3'$ the space of infinitesimal automorphisms fixing $0$ is trivial. The map $H_{3,t}: (z,w) \mapsto (z, t z, (1+t^2) w)$ for $t \in \mathbb R$ embeds $M_3$ into $M_3'$, is transversal and $2$-nondegenerate at $0$ for $t \neq 0$. This shows that the map $H_{3,t}: M_3 \rightarrow M_3'$ is not locally rigid. If we solve the linear equation \eqref{lineq} in this case we obtain that if $t\neq 0$ then $\dim_{\mathbb R} \aut_0( H_{3,t}) = 10$.
\end{example}

\section{Genericity of Local Rigidity of Embeddings into Hyperquadrics}\label{sec:genericity}
\begin{theorem}\label{genericity}
 There exist integers $K, N(K)$ and an algebraic subvariety $X \subset \mathbb C^{N(K)}$ such that if we let $H: \mathbb C^2 \to \mathbb C^3$ be a germ of a holomorphic map, which is transversal and $2$-nondegenerate at $0$, and satisfies $H(0)=0$ and $H(M) \subset \mathbb H^3_\varepsilon$, where $M$ is given as in \eqref{oftheform}, we have if $j_0^K F$ belongs to the complement of $X$ then $H$ is locally rigid.
\end{theorem}

\begin{remark}
The estimate for $K$ we would obtain from the proof of Theorem \ref{genericity} is very rough. We have computed when $F(z,w) = F(z)$ in \eqref{oftheform} and $M' = \mathbb H^3$ then $K=8$. Even in this case we do not know whether $X$ is trivial or not.
\end{remark}

\begin{proof}
Let $M$ be given as in \eqref{oftheform} and the embedding of the form $H: (z,w) \mapsto (z,F(z,w),w)$ with $F_{z^2}(0)\neq0$, for which we consider the system of equations given in the model case $F(z,w) = z^2$ in Example \ref{modelexample}. First we keep track which order $K$ of $j_0^K F$ is involved in the equations given in the model case. In order to use the techniques of Proposition \ref{linjetparam} we first need to perform a change of coordinates such that $M = \{ {\rm Im} w - |z|^2 -\varepsilon |F(z,w)|^2 = 0\}$, where $F: \mathbb C^2 \rightarrow \mathbb C$ is holomorphic and $F(0)=0$, is given in normal coordinates, $M=\{w - Q(z,\chi,\tau)=0 \}$, where $Q(z,0,\tau) \equiv Q(0,\chi,\tau) \equiv \tau$. We want to see how $j_0^k Q$ depends on $j_0^{k'} F$, so we briefly look into the details of the aforementioned well-known change of coordinates. First we rewrite the original defining function for $M$ from \eqref{oftheform} as $\rho''(z,w,\chi,\tau) = w - \widetilde Q(z,\chi,\tau)$ after an application of the implicit function theorem. Note, since $F(0)=0$, $j_0^k\widetilde Q$ depends on $j_0^{k-1} F$. We seek for a biholomorphism $(z,w) \mapsto(z,w + i g(z,w))$, where $g=O(2)$ and $g(0,w)=\overline g(0,w)$. We write $\rho'(z,w,\chi,\tau)=\rho''(z,w+ i g(z,w),\chi, \tau - i \overline g(\chi,\tau))$ and require $\rho'(z,w,0,w) =0$. This holds if and only if $i g(z,w) + i g(0,w) - \widetilde Q(z,0,w - i g(0,w)) = 0$, hence $g_{z^m}(0,w) = - i \widetilde Q(0,0,w - i g(0,w))$ for $m \geq 1$ and if we set $z=0$ we solve in $2 i g(0,w) - \widetilde Q(0,0,w- i g(0,w)) = 0$ for $g(0,w)$ by the implicit function theorem. Thus $j_0^k g$ depends on $j_0^{k-1} F$. Finally we solve for $w$ in $\rho'(z,w,\chi,\tau)$ again by the implicit function theorem, to obtain normal coordinates for $M$, such that $M=\{w - Q(z,\chi,\tau)=0 \}$, and $Q$ has the required properties. In particular we obtain that $j_0^k Q$ depends on $j_0^{k-1} F$. 

We are now in the situation to apply Proposition \ref{linjetparam}. Inspecting the proof of Proposition \ref{linjetparam} (i) we obtain that in order to compute the coefficient of $ z^m w^k$ in $\Psi_{\ell}$ from \eqref{formalpsi} we need to consider the coefficient of $z^{2 k + m} t^k$ in expressions of the following form
\begin{align*}
\Psi^{h,j_1,j_2}_{\ell}(z,t) & = \phi_{\ell}^{h,j_1,j_2}(z, A_1(z) \psi(z,t)) \\
& = \frac{1}{s{\overline s}^3(0)}  q_\ell^{h,j_1,j_2}( \partial^4\overline r^k(H,\overline H), \partial^4  r^k(H,\overline H), \partial^4 Q,\partial^4 \bar Q),
\end{align*}
where the functions  $\partial^4\overline r^k(H,\overline H)$, 
 $ \partial^4  r^k(H,\overline H) $, 
  $ \partial^4 Q$, and $ \partial^4 \bar Q$ are either evaluated
   along $(z,Q(z,A_1(z) \psi(z,t),0),A_1(z) \psi(z,t),0)$ or $(A_1(z) \psi(z,t),0,0,0)$. 
   If we consider the linear system from the model case in this general situation given by $S(\Psi,\rho)$, which is linear in $\lambda_0$ from the model case, we notice that the highest order $K \in \mathbb N$ of derivatives of $\Psi^{h,j_1,j_2}_{\ell}(z,t)$ occurs in $\Psi_3^{-1,7}$ from $S_1(\Psi)$. Also note that for the expressions from $S_3(\rho)$ we need to take derivatives of order less than $K$. Thus $S(\Psi,\rho)$ depends on $j_0^{K-1} F$ and its conjugates and moreover this dependence is polynomial. Thus if we compute the Jacobian matrix of $S(\Psi,\rho)$ with respect to $\lambda_0$, the resulting determinant $d(j_0^{K-1} F,j_0^{K-1} \overline F)$ is a polynomial. Hence there is an integer $N(K)$ such that $X=\{d=0\} \subset \mathbb C^{N(K)}$ is an algebraic subvariety, such that if $(j_0^{K-1} F, j_0^{K-1} \overline F) \in \mathbb C^{N(K)}$ does not belong to $X$, by Theorem \ref{suffcon2} we have local rigidity of $H$, which proves the theorem.
\end{proof}

\section{The Sphere Case}\label{spherecase} 
In this section we would like to show that the condition given in Theorem \ref{suffcon2} is not necessary for local rigidity. For the sphere case we have shown the properness of the action of isotropies on transversal and $2$-nondegenerate maps in \cite[Theorem 1.3]{Re2}, which corresponds to Lemma \ref{proact} for spheres, and the freeness of $G'$ in the sphere case is given in Lemma \ref{free}. Thus with the same proof we obtain that Theorem \ref{suffcon2} also holds when we consider $M=\mathbb H^2$ and $M'=\mathbb H^3$. From \cite{Re} it follows that the mapping 
\begin{align*}
H(z,w) = \frac{\left(z(1 + i w), \sqrt 2 z^2, w \right)}{1-w^2},
\end{align*}
which is a scaled version of the map $G_1^+$ from \cite[Theorem 1.4]{Re}, is locally rigid. This map corresponds to $(z,w)\mapsto (z^2, \sqrt{2} z w, w^2)$ as a map from $\mathbb S^2$ to $\mathbb S^3$. We want to compute $\aut_0 (H)$, which requires to solve the following equation:
\begin{align}
\label{rhosphere}
  \alpha_{3}(z,w) - \overline \alpha_{3}(\chi,\tau) - 2 i & \left(  \alpha_{1}(z,w)\left( \frac{\chi(1 - i \tau)}{1-\tau^2}\right) + \overline \alpha_{1}(\chi,\tau) \left(\frac{z(1+ i w)}{1-w^2} \right)\right. \\ \nonumber
 &   \left.+  \alpha_{2}(z,w) \left( \frac{\sqrt 2 \chi^2}{1-\tau^2}\right) + \overline \alpha_{2}(\chi,\tau) \left(\frac{\sqrt{2} z^2}{1-w^2} \right)\right) = 0,
\end{align}
if $w = \tau + 2 i z \chi$, for $(\alpha_1(z,w),\alpha_2(z,w),\alpha_3(z,w))\in (\mathbb C\{z,w\})^3$ with $\alpha_j(0,0)=0$. We proceed as in Example \ref{modelexample} and use the notation from there, we set $\Psi=(\Psi_1,\Psi_2,\Psi_3) \in (\mathbb C[[z,w]])^3$ from \eqref{formalpsi} for the jet parametrization. We give the equations deduced in Proposition \ref{linjetparam}. The first set of coefficients we denote by $S_1(\Psi)$  and are given by Proposition \ref{linjetparam} (i) as follows:
\begin{align*}
\Psi_1^{-1,4} & = -12 \Lambda_2^{0,1} + \Lambda_2^{0,3},\\
\Psi_3^{-1,4} & = 12 \sqrt{2} \Lambda_1^{0,1} - \sqrt{2} \Lambda_1^{0,3} - 12 i \Lambda_2^{1,0} + 3 i \Lambda_2^{1,2},\\
\Psi_1^{-1,5} & = 60 \Lambda_2^{0,1} + 6 i \Lambda_2^{0,2} - 5 \Lambda_2^{0,3} - i \Lambda_2^{0,4},\\
\Psi_3^{-1,5} & = -24 i \Lambda_1^{0,1} + 24 \Lambda_1^{0,2} + 2 i  \Lambda_1^{0,3} - \Lambda_1^{0,4} - 6 \sqrt{2} i \Lambda_2^{1,1} + \sqrt{2} i \Lambda_2^{1,3}.
\end{align*}
For Proposition \ref{linjetparam} (ii) we obtain the following equations from $\Psi$ denoted by $S_2(\Psi)$:
\begin{align*}
2 \Psi_2^{2,1} - \Lambda_2^{2,1} & = 12 \Lambda_1^{1,0} - 6 \Lambda_1^{1,2} + 3 \sqrt{2} i \Lambda_2^{2,1} - 12 \Lambda_3^{0,1} + \Lambda_3^{0,3},\\
4 \Psi_1^{2,2} - \Lambda_1^{2,2} & = 6 \Lambda_1^{2,0} + 6 i \Lambda_1^{2,1} - 3 \Lambda_1^{2,2} - 24 \Lambda_3^{1,1} + 2 \Lambda_3^{1,3},\\
4 \Psi_2^{2,2} - \Lambda_2^{2,2} & = -48 \Lambda_1^{1,0} - 24 i \Lambda_1^{1,1} + 24 \Lambda_1^{1,2} + 2 i \Lambda_1^{1,3} - 6 \sqrt{2} i \Lambda_2^{2,1} + 24 \Lambda_3^{0,1} + 24 i \Lambda_3^{0,2}\\& \quad  - 2 \Lambda_3^{0,3} - i \Lambda_3^{0,4},\\
6 \Psi_1^{1,3} - \Lambda_1^{1,3} & = - 96 i \Lambda_1^{1,0} + 84 \Lambda_1^{1,1} + 48 i \Lambda_1^{1,2} - 8 \Lambda_1^{1,3} - 6 \sqrt{2} i \Lambda_2^{2,0} + 9 \sqrt{2} \Lambda_2^{2,1} \\ & \quad + 3 \sqrt{2} i \Lambda_2^{2,2} + 36 i \Lambda_3^{0,1} - 72 \Lambda_3^{0,3} + 3 \Lambda_3^{0,4},\\
6 \Psi_2^{1,3} - \Lambda_2^{1,3} & = -2 88 \Lambda_1^{0,1} - 72 i \Lambda_1^{0,2} + 24 \Lambda_1^{0,3} + 3 i \Lambda_1^{0,4} + 96 \sqrt{2} i \Lambda_2^{1,0} - 12 \sqrt{2} \Lambda_2^{1,1}\\ & \quad - 24 \sqrt{2} i \Lambda_2^{1,2} + \sqrt{2} \Lambda_2^{1,3}.
\end{align*}
From Proposition \ref{linjetparam} (iii) we obtain first the following set $S_{3.1}(\rho)$ of coefficients of $\rho(z,\chi,\tau)$, which denotes the left-hand side of \eqref{rhosphere}:
\begin{align*}
\rho_{0,1,1} & = 2 \Lambda_1^{0,1} - i\overline{ \Lambda}_3^{1,1},\\
\rho_{1,1,0} & = \Lambda_3^{0,1} - \Lambda_1^{1,0} - \overline{\Lambda}_1^{1,0},\\
\rho_{1,2,0} & = 4 i \Lambda_1^{0,1} +2 \sqrt{2} \Lambda_2^{1,0} + \overline{\Lambda}_1^{2,0},\\
\rho_{1,2,1} & = 4 \Lambda_1^{0,1} + 4 i \Lambda_1^{0,2} + \sqrt{2} \Lambda_2^{1,1} + i \overline{\Lambda}_1^{2,0} + \overline{\Lambda}_1^{2,1},\\
\rho_{1,3,0} & = 4 \sqrt{2} i \Lambda_2^{0,1} +\overline{\Lambda}_3^{2,1},\\
\rho_{2,2,0} & = - 4 i \Lambda_1^{1,1} - \sqrt{2} \Lambda_2^{2,0} + 4 \overline{\Lambda}_1^{1,0} - \sqrt{2} \overline{\Lambda}_2^{2,0} + 2 i \Lambda_3^{0,2},\\
\rho_{2,2,1} & = -4 \Lambda_1^{1,1} - 4 i \Lambda_1^{1,2} - \sqrt{2} \Lambda_2^{2,1} + 4 i \overline{\Lambda}_1^{1,0} + 4 \overline{\Lambda}_1^{1,1} - 6 i \overline{\Lambda}_1^{1,2} +2 \sqrt{2} \overline{\Lambda}_2^{2,1} - 12 i \overline{\Lambda}_3^{0,1} + i \overline{\Lambda}_3^{0,3} + i \Lambda_3^{0,3},\\
\rho_{2,3,0} & = 2 i \Lambda_1^{0,2} + \sqrt{2} \Lambda_2^{1,1} + i \overline{\Lambda}_1^{2,0} + \overline{\Lambda}_1^{2,1} - \overline{\Lambda}_3^{1,2},\\
\rho_{2,3,1} & =  -96i \Lambda_1^{0,1} + 4 \Lambda_1^{0,2} + 10 i \Lambda_1^{0,3} - 48 \sqrt{2} \Lambda_2^{1,0} + 14 \sqrt{2} \Lambda_2^{1,2} + 6 \overline{\Lambda}_1^{2,0} - 2 i \overline{\Lambda}_1^{2,1} - \overline{\Lambda}_1^{2,2} - 12 \overline{\Lambda}_3^{1,1} + \overline{\Lambda}_3^{1,3},\\
\rho_{2,4,0} & = \sqrt{2} i \Lambda_2^{0,2} + \overline{\Lambda}_3^{2,2},
\end{align*}
and the following coefficients of $\rho$, which we denote by $S_{3.2}(\rho)$:
\begin{align*}
\rho_{1,1,1} & = i \Lambda_1^{1,0} - i \overline{\Lambda}_1^{1,0} - \Lambda_1^{1,1}  - \overline{\Lambda}_1^{1,1} + \Lambda_3^{0,3},\\
\rho_{2,2,2} & = -8 i \Lambda_1^{1,1} - 8 \Lambda_1^{1,2} - 2 i \Lambda_1^{1,3} - 2 \sqrt{2} \Lambda_2^{2,0} - \sqrt{2} \Lambda_2^{2,2} + 48 \overline{\Lambda}_1^{1,0} - 40 i \overline{\Lambda}_1^{1,1} - 8 \overline{\Lambda}_1^{1,2} + 2 i \overline{\Lambda}_1^{1,3}  \\ & \quad - 2 \sqrt{2} \overline{\Lambda}_2^{2,0} - \sqrt{2} \overline{\Lambda}_2^{2,2} + 24 i \overline{\Lambda}_3^{0,2} - i \overline{\Lambda}_3^{0,4} + i \Lambda_3^{0,4}.
\end{align*}
We write 
\begin{align*}
S(\Psi,\rho)=\{S_1(\Psi),\overline{S}_1(\Psi),S_2(\Psi),\overline{S}_2(\Psi),S_{3.1}(\rho),\overline{S}_{3.1}(\rho), S_{3.2}(\rho)\},
\end{align*}
which consists of $40$ components. Then we consider $\lambda \in \mathbb C^{40}$, the vector consisting of the following elements of the $4$-jet of $\alpha$ at $0$, given by
\begin{align*}
\Lambda_{3}^{0,1}, \Lambda_{1}^{2,0}, \Lambda_{3}^{1,1}, \Lambda_{3}^{0,2},\Lambda_{1}^{2,1} ,\Lambda_{3}^{2,1},\Lambda_{1}^{1,2} ,\Lambda_{3}^{1,2},\Lambda_{1}^{2,2} ,\Lambda_{3}^{2,2},  \Lambda^{0,3}, \Lambda^{1,3}, \Lambda^{0,4},
\end{align*}
and its conjugates and the derivatives $\overline{\Lambda}_{2}^{2,2}$ and $\overline{\Lambda}_{2}^{2,0}$. The remaining elements in $J_0^4$ do not occur. We have used the notation $\Lambda^{m,n}=(\Lambda_{1}^{m,n},\Lambda_{2}^{m,n},\Lambda_{3}^{m,n})$. It holds that the Jacobian of $S(\Psi,\rho)$ with respect to $\lambda$ is of full rank. After we get rid of trivial solutions of \eqref{rhosphere} by considering infinitesimal automorphisms of the spheres fixing $0$, we end up with the following $8$ infinitesimal deformations in the space $\mathfrak {hol}_0 (H)$ of solutions of \eqref{rhosphere}:
\begin{align*}
X_1 & = \frac{\sqrt{2} w z}{(w+i) \left(w^2-1\right)} \frac{\partial}{\partial z_1'} + \frac{(w-i) z^2}{(w+i) \left(w^2-1\right)} \frac{\partial}{\partial z_2'} \\
X_2  & =-\frac{w z^2}{(w+i) \left(w^2-1\right)}  \frac{\partial}{\partial z_1'} +  \frac{i z \left(w^2+i w+2 z^2\right)}{\sqrt{2} (w+i)
   \left(w^2-1\right)} \frac{\partial}{\partial z_2'}  \\
X_3 & = \frac{3 w^2 z}{\sqrt{2} (w+i) \left(w^2-1\right)}  \frac{\partial}{\partial z_1'} -\frac{3 z^2}{(w+i) \left(w^2-1\right)}  \frac{\partial}{\partial z_2'}\\
X_4 & =  \frac{w \left(w^2-4 z^2-1\right)}{2 (w+i) \left(w^2-1\right)}  \frac{\partial}{\partial z_1'} + \frac{i z \left(w^2+2 i w+4
   z^2+1\right)}{\sqrt{2} (w+i) \left(w^2-1\right)}  \frac{\partial}{\partial z_2'} + \frac{w z}{w^2-1}  \frac{\partial}{\partial w'}\\
X_5  & =   \frac{4 i \sqrt{2} w z^3}{(w+i) \left(w^2-1\right)^2}  \frac{\partial}{\partial z_1'}  - \frac{i w^5-w^4-i w^3+w^2-4 i z^4 w-4 z^4}{(w+i)
   \left(w^2-1\right)^2}   \frac{\partial}{\partial z_2'} +   \frac{2 \sqrt{2} w^2 z^2}{\left(w^2-1\right)^2} \frac{\partial}{\partial w'}\\
X_6 & = -\frac{2 \sqrt{2} w z^3}{(w+i) \left(w^2-1\right)^2}  \frac{\partial}{\partial z_1'}-\frac{w^5+i w^4-w^3-i w^2+4 z^4 w-4 i z^4}{2 (w+i)
   \left(w^2-1\right)^2} \frac{\partial}{\partial z_2'} \\
   &\quad + \frac{i \sqrt{2} w^2 z^2}{\left(w^2-1\right)^2}\frac{\partial}{\partial w'} \\
X_7 & =   \frac{w^4+4 i z^2 w^3-\left(2 z^2+1\right) w^2+2 z^2}{(w+i) \left(w^2-1\right)^2}  \frac{\partial}{\partial z_1'} \\ & \quad + \frac{\sqrt{2} z
   \left(-w^4-i w^3+2 z^2 w^2+i \left(2 z^2+1\right) w+1\right)}{(w+i) \left(w^2-1\right)^2}  \frac{\partial}{\partial z_2'} + \frac{2 w^2
   z}{\left(w^2-1\right)^2} \frac{\partial}{\partial w'} \\
X_8 & =  \frac{w^2 \left(w^2+4 i z^2 w-1\right)}{2 (w+i) \left(w^2-1\right)^2}   \frac{\partial}{\partial z_1'} + \frac{z \left(-w^4-2 i w^3+\left(4
   z^2+1\right) w^2+2 i \left(z^2+1\right) w-2 z^2\right)}{\sqrt{2} (w+i) \left(w^2-1\right)^2} \frac{\partial}{\partial z_2'} \\
   & \quad + \frac{w^2
   z}{\left(w^2-1\right)^2} \frac{\partial}{\partial w'}.
\end{align*}
Hence, this example shows that the condition given in Theorem \ref{suffcon2} is not necessary for local rigidity.


\begin{bibdiv}
\begin{biblist}

\bib{BER}{article}{
   author={Baouendi, M. S.},
   author={Ebenfelt, P.},
   author={Rothschild, L. P.},
   title={Parametrization of local biholomorphisms of real analytic
   hypersurfaces},
   journal={Asian J. Math.},
   volume={1},
   date={1997},
   number={1},
   pages={1--16},
   issn={1093-6106},
   review={\MR{1480988 (99b:32022)}},
}

\bib{BER2}{book}{
   author={Baouendi, M. S.},
   author={Ebenfelt, P.},
   author={Rothschild, L. P.},
   title={Real submanifolds in complex space and their mappings},
   series={Princeton Mathematical Series},
   volume={47},
   publisher={Princeton University Press},
   place={Princeton, NJ},
   date={1999},
   pages={xii+404},
   isbn={0-691-00498-6},
   review={\MR{1668103 (2000b:32066)}},
}

\bib{BV}{article}{
   author={Beloshapka, V. K.},
   author={Vitushkin, A. G.},
   title={Estimates of the radius of convergence of power series that give
   mappings of analytic hypersurfaces},
   language={Russian},
   journal={Izv. Akad. Nauk SSSR Ser. Mat.},
   volume={45},
   date={1981},
   number={5},
   pages={962--984, 1198},
   issn={0373-2436},
   review={\MR{637612 (83f:32017)}},
}

\bib{CM}{article}{
   author={Chern, S. S.},
   author={Moser, J. K.},
   title={Real hypersurfaces in complex manifolds},
   journal={Acta Math.},
   volume={133},
   date={1974},
   pages={219--271},
   issn={0001-5962},
   review={\MR{0425155 (54 \#13112)}},
}

\bib{CH}{article}{
   author={Cho, Chung-Ki},
   author={Han, Chong-Kyu},
   title={Finiteness of infinitesimal deformations of CR mappings of CR
   manifolds of nondegenerate Levi form},
   journal={J. Korean Math. Soc.},
   volume={39},
   date={2002},
   number={1},
   pages={91--102},
   issn={0304-9914},
   review={\MR{1872584 (2002j:32036)}},
   doi={10.4134/JKMS.2002.39.1.091},
}

\bib{Da}{article}{
   author={D'Angelo, John P.},
   title={Proper holomorphic maps between balls of different dimensions},
   journal={Michigan Math. J.},
   volume={35},
   date={1988},
   number={1},
   pages={83--90},
   issn={0026-2285},
   review={\MR{931941 (89g:32038)}},
   doi={10.1307/mmj/1029003683},
}

\bib{DK}{book}{
   author={Duistermaat, J. J.},
   author={Kolk, J. A. C.},
   title={Lie groups},
   series={Universitext},
   publisher={Springer-Verlag, Berlin},
   date={2000},
   pages={viii+344},
   isbn={3-540-15293-8},
   review={\MR{1738431 (2001j:22008)}},
   doi={10.1007/978-3-642-56936-4},
}

\bib{Ez} {article}{
    AUTHOR = {Ezhov, V. V.},
     TITLE = {Linearization of the stability group of a class of
              hypersurfaces},
   JOURNAL = {Uspekhi Mat. Nauk},
  FJOURNAL = {Akademiya Nauk SSSR i Moskovskoe Matematicheskoe Obshchestvo.
              Uspekhi Matematicheskikh Nauk},
    VOLUME = {41},
      YEAR = {1986},
    NUMBER = {3(249)},
     PAGES = {181--182},
      ISSN = {0042-1316},
   MRCLASS = {32F25},
  MRNUMBER = {854251 (87k:32033)},
MRREVIEWER = {Harold P. Boas},
}

\bib{Fa2}{article}{
   author={Faran, James J.},
   title={Maps from the two-ball to the three-ball},
   journal={Invent. Math.},
   volume={68},
   date={1982},
   number={3},
   pages={441--475},
   issn={0020-9910},
   review={\MR{669425 (83k:32038)}},
   doi={10.1007/BF01389412},
}

\bib{Fa}{article}{
   author={Faran, James J.},
   title={The linearity of proper holomorphic maps between balls in the low
   codimension case},
   journal={J. Differential Geom.},
   volume={24},
   date={1986},
   number={1},
   pages={15--17},
   issn={0022-040X},
   review={\MR{857373 (87k:32050)}},
}


\bib{Hu}{article}{
   author={Huang, Xiaojun},
   title={On a linearity problem for proper holomorphic maps between balls
   in complex spaces of different dimensions},
   journal={J. Differential Geom.},
   volume={51},
   date={1999},
   number={1},
   pages={13--33},
   issn={0022-040X},
   review={\MR{1703603 (2000e:32020)}},
}

\bib{HJ}{article}{
   author={Huang, Xiaojun},
   author={Ji, Shanyu},
   title={Mapping $\bold B^n$ into $\bold B^{2n-1}$},
   journal={Invent. Math.},
   volume={145},
   date={2001},
   number={2},
   pages={219--250},
   issn={0020-9910},
   review={\MR{1872546 (2002i:32013)}},
   doi={10.1007/s002220100140},
}

\bib{JL}{article}{
   author={Juhlin, Robert},
   author={Lamel, Bernhard},
   title={Automorphism groups of minimal real-analytic CR manifolds},
   journal={J. Eur. Math. Soc. (JEMS)},
   volume={15},
   date={2013},
   number={2},
   pages={509--537},
   issn={1435-9855},
   review={\MR{3017044}},
   doi={10.4171/JEMS/366},
}

\bib{KL}{article}{
   author={Kruzhilin, N. G.},
   author={Loboda, A. V.},
   title={Linearization of local automorphisms of pseudoconvex surfaces},
   language={Russian},
   journal={Dokl. Akad. Nauk SSSR},
   volume={271},
   date={1983},
   number={2},
   pages={280--282},
   issn={0002-3264},
   review={\MR{718188 (85c:32052)}},
}

\bib{La}{article}{
   author={Lamel, Bernhard},
   title={Holomorphic maps of real submanifolds in complex spaces of
   different dimensions},
   journal={Pacific J. Math.},
   volume={201},
   date={2001},
   number={2},
   pages={357--387},
   issn={0030-8730},
   review={\MR{1875899 (2003e:32066)}},
   doi={10.2140/pjm.2001.201.357},
}

\bib{Le}{article}{
   author={Lebl, Ji{\v{r}}{\'{\i}}},
   title={Normal forms, Hermitian operators, and CR maps of spheres and
   hyperquadrics},
   journal={Michigan Math. J.},
   volume={60},
   date={2011},
   number={3},
   pages={603--628},
   issn={0026-2285},
   review={\MR{2861091}},
   doi={10.1307/mmj/1320763051},
}

\bib{Re}{article}{
   author={Reiter, Michael},
   title={Holomorphic mappings of hyperquadrics from $\mathbb C^2$ to $\mathbb C^3$},
   journal={Ph.D. thesis, University of Vienna, http://othes.univie.ac.at/33603/},
   date={2014}
   }
   
 \bib{Re2}{article}{
   author={Reiter, Michael},
   title={Classification of holomorphic mappings of hyperquadrics from $\mathbb C^2$ to $\mathbb C^3$},
   note={To appear in {\em Journal of Geometric Analysis}, \href{http://dx.doi.org/10.1007/s12220-015-9594-6}{doi:10.1007/s12220-015-9594-6}},
   date={2014},
   }  
   
\bib{Re3}{article}{
   author={Reiter, Michael},
   title={Topological Aspects of holomorphic mappings of hyperquadrics from $\mathbb C^2$ to $\mathbb C^3$},
   journal={Submitted, http://arxiv.org/abs/1410.6262},
   date={2014},
   }

\bib{We}{article}{
   author={Webster, S. M.},
   title={The rigidity of C-R hypersurfaces in a sphere},
   journal={Indiana Univ. Math. J.},
   volume={28},
   date={1979},
   number={3},
   pages={405--416},
   issn={0022-2518},
   review={\MR{529673 (80d:32022)}},
   doi={10.1512/iumj.1979.28.28027},
}


 \bib{wolfram}{article}{
   author={Wolfram Research},
   title={ Mathematica 9.0.1.0},
   journal={ Wolfram Research, Inc., Champaign, Illinois},
   date={2013},
   }  


\bib{Ebenfelt:2004um}{article}{
author = {Ebenfelt, Peter},
author = {Huang, Xiaojun},
author = {Zaitsev, Dmitri},
title = {{Rigidity of CR-immersions into spheres}},
journal = {Communications in Analysis and Geometry},
year = {2004},
volume = {12},
number = {3},
pages = {631--670}
}

\bib{Baouendi:2008cz}{article}{
author = {Baouendi, M Salah},
author = {Ebenfelt, Peter},
author = {Huang, Xiaojun},
title = {{Super-rigidity for CR embeddings of real hypersurfaces into hyperquadrics}},
journal = {Adv. Math.},
year = {2008},
volume = {219},
number = {5},
pages = {1427--1445},
}

\bib{Baouendi:2005uq}{article}{
author = {Baouendi, M Salah},
author = {Huang, Xiaojun},
title = {{Super-rigidity for holomorphic mappings between hyperquadrics with positive signature}},
journal = {Journal of Differential Geometry},
year = {2005},
volume = {69},
number = {2},
pages = {379--398}
}

\bib{Ebenfelt:2005wj}{article}{
author = {Ebenfelt, Peter},
author = {Huang, Xiaojun},
author = {Zaitsev, Dmitri},
title = {{The equivalence problem and rigidity for hypersurfaces embedded into hyperquadrics}},
journal = {American Journal of Mathematics},
year = {2005},
volume = {127},
number = {1},
pages = {169--191}
}

\bib{Ebenfelt:2015wq}{article}{
author = {Ebenfelt, Peter},
author = {Shroff, Ravi},
title = {{Partial rigidity of CR embeddings of real hypersurfaces into hyperquadrics with small signature difference}},
journal = {Communications in Analysis and Geometry},
year = {2015},
volume = {23},
number = {1},
pages = {159--190}
}




\end{biblist}
\end{bibdiv}
\end{document}